\documentclass{amsart}
\usepackage[margin=1in]{geometry}
\usepackage{etex} 

\usepackage{amsthm,amssymb,amsmath,latexsym, amscd, mathtools}
\input cyracc.def

\usepackage{thmtools}
\usepackage{thm-restate}

\usepackage{tikz}
\usepackage{tikz-cd}
\usepackage{caption}
\usetikzlibrary{decorations.markings}

\newtheorem{theorem}{Theorem}[section]
\newtheorem{lemma}[theorem]{Lemma}
\newtheorem{proposition}[theorem]{Proposition}
\newtheorem{definition}[theorem]{Definition}

\newtheorem{corollary}[theorem]{Corollary}
\newtheorem{intromain}{Theorem}
\newtheorem{textmain}{Theorem}
\theoremstyle{remark}
\newtheorem{remark}[theorem]{Remark}
\newtheorem{example}[theorem]{Example}
\newtheorem{problem}[theorem]{Problem}
\newtheorem{question}[theorem]{Question}
\newtheorem{construction}[theorem]{Construction}

\newcommand{\QQ}{\mathbb{Q}}
\newcommand{\RR}{\mathbb{R}}
\renewcommand{\SS}{\mathbb{S}} 

\newcommand{\ZZ}{\mathbb{Z}}

\newcommand{\vA}{\mathcal{A}}
\newcommand{\vB}{\mathcal{B}}

\newcommand{\vD}{\mathcal{D}}

\newcommand{\vG}{\mathcal{G}}

\newcommand{\vL}{\mathcal{L}}

\newcommand{\vR}{\mathcal{R}}

\newcommand{\fI}{\mathfrak{I}}

\newcommand{\fL}{\mathfrak{L}}

\newcommand{\fQ}{\mathfrak{Q}}

\newcommand{\id}{\text{id}}
\newcommand{\lra}{\longrightarrow}

\newcommand{\ol}[1]{\overline{#1}}

\newcommand{\sm}{\setminus}
\newcommand{\wt}[1]{\widetilde{#1}}

\newcommand{\LOs}{\mathbf{LO_*}}
\newcommand{\BigL}{\mathbf{BigLO_*}}
\newcommand{\Circ}{\mathbf{Circ}}
\newcommand{\BigC}{\mathbf{BigCirc}}

\title[Amalgamating circularly-ordered groups]{Free products of circularly-ordered groups with amalgamated subgroup}

\date{\today}

\author[Adam Clay]{Adam Clay}
\address{Department of Mathematics\\
University of Manitoba \\
Winnipeg \\
MB Canada R3T 2N2} \email{Adam.Clay@umanitoba.ca}
\urladdr{http://server.math.umanitoba.ca/~claya/} 

\author[Tyrone Ghaswala]{Tyrone Ghaswala}
\address{Department of Mathematics\\
University of Manitoba \\
Winnipeg \\
MB Canada R3T 2N2}
\email{ty.ghaswala@gmail.com}
\urladdr{http://server.math.umanitoba.ca/~ghaswalt/}

\begin{document}

\maketitle

\begin{abstract}
This paper gives necessary and sufficient conditions that the free product with amalgamation of circularly-ordered groups admit a circular ordering extending the given orderings of the factors.  Our result follows from establishing a categorical framework that allows the problem to be restated in terms of amalgamating certain left-ordered central extensions, where we are able to apply work of Bludov and Glass.  
\end{abstract}

\section{Introduction}

A group $G$ is \emph{left-orderable} if its elements admit a strict total ordering that is invariant under multiplication from the left.   A group is called \emph{circularly-orderable} if $G$ admits an \emph{orientation cocycle} $c: G^3 \rightarrow \{ 0, \pm1 \}$ that is invariant under multiplication from the left.  When $G$ is countable, these properties are equivalent to admitting faithful order-preserving actions by homeomorphisms on $\RR$ and $\SS^1$ respectively.

Understanding the behaviour of left orderability and circular orderability with respect to various group-theoretic constructions (such as direct products, extensions, free products and free products with amalgamation) is one of the basic questions which has, at times, proved to be an obstacle to a number of applications.  For example, it was not until recent years that Bludov and Glass  \cite[Theorem A]{BV09} provided necessary and sufficient conditions that the free product with amalgamation of a family of left-ordered groups be left-orderable, and that it admit an ordering extending the given orderings of the factors.   Their work was readily applied to the solvability of the word problem in left-orderable groups \cite[Theorem E]{BV09}, was used to left-order the fundamental groups of many $3$-manifolds \cite{BC17}\cite{CLW13}\cite{CM13}, and was also extended to give conditions that an arbitrary graph of groups with left-orderable vertex groups be left-orderable \cite{Chiswell11}.  

This paper further builds on the results of Bludov and Glass to determine necessary and sufficient conditions that the free product with amalgamation of an arbitrary family of circularly-ordered groups be circularly-orderable, with an ordering that extends the given orderings of the factors.  Such conditions are given in Propositions \ref{bg-circ-forward} and \ref{bg-circ-back}.

Our approach is to observe that certain classical lifting and quotient constructions, which allow one to pass from a circularly-ordered group $(G, c)$ to a left-ordered cyclic central extension $\widetilde{G}$ (and vice versa), behave functorially.  In fact, these two functors provide an equivalence between appropriately defined categories $\LOs$ and $\Circ$ of left-ordered and circularly-ordered groups respectively.  Under such a setup, one might expect that amalgamated free products of circularly-ordered groups will correspond to certain colimits in $\Circ$, and that these colimits would be carried via the categorical equivalence to colimits in $\LOs$---where one can then apply the results of Bludov and Glass to construct a left ordering that descends to a circular ordering on the original amalgamated free product.

While this is roughly the correct idea, it turns out that the categories $\LOs$ and $\Circ$ do not admit colimits (though it does turn out that each is a tensor category when equipped with a certain colimit-like operation defined in \cite{BS15}).  We therefore embed these categories in larger categories where the desired colimits exist, allowing us to pursue the line of proof above on sound mathematical footing.   As a result, we give necessary and sufficient conditions that a free product with amalgamation of circularly-ordered groups $\{(G_i, c_i)\}_{i \in I}$ be circularly-ordered, by examining certain left-ordered cyclic central extensions $\{(\widetilde G_i, <_{c_i}) \}_{i \in I}$.  If $H_i \subset G_i$ is a subgroup for each $i \in I$ and $\phi_i :H \rightarrow H_i$ are order-preserving isomorphisms with a circularly-ordered group $(H, d)$ for all $i \in I$, then these amalgamating isomorphisms will lift to give $\widetilde \phi_i : \widetilde{H} \rightarrow \widetilde H_i$, as the lifting construction is functorial.  We prove:

\begin{intromain}
\label{intro-main-theorem}
Suppose $(G_i,c_i)$ are circularly-ordered groups for $i \in I$, each equipped with a subgroup $H_i \subset G_i$ and an order-preserving isomorphism $\phi_i:(H,d) \to (H_i,c_i)$ from a circularly-ordered group $(H,d)$.  The following are equivalent:
\begin{enumerate}
\item The group $*_{i \in I} G_i  (H_i \stackrel{\phi_i}{\cong} H)$ admits a circular ordering $c$ which extends the orderings $c_i$ of $G_i$ for all $i \in I$.
\item The group $*_{i \in I} \widetilde G_i  (\widetilde H_i \stackrel{\widetilde\phi_i}{\cong}\widetilde{H})$ admits a left ordering $<$ which extends each of the left orderings $<_{c_i}$ of $\widetilde G_i$ for $i \in I$.
\end{enumerate}
\end{intromain}

One can also approach this theorem with a more narrow view, by setting up an isomorphism between a quotient of $*_{i\in I} \widetilde G_i (\widetilde{H}_i \stackrel{\widetilde \phi_i}{\cong}\widetilde{H})$ and the group $*_{i \in I} G_i  (H_i \stackrel{\phi_i}{\cong} H)$, and then checking that the isomorphism restricts appropriately to certain subgroups identified with lifts of the factors (see Remark \ref{sketch-of-proof}).  While direct, this approach masks the fact that the group isomorphism exists for general categorical reasons, and fails to uncover the additional data yielded by the larger categories that contain colimits of diagrams in $\LOs$ and $\Circ$:   In the course of constructing the colimit corresponding to a collection of circularly-ordered groups $(G_i, c_i)$ with subgroups $H_i$ identified as above, we find that the colimit carries data (in the form of a collection of $2$-cocyles) that encode information about all possible extensions of the circular orderings $\{c_i\}_{i \in I}$ to a circular ordering of $*_{i \in I} G_i  (H_i \stackrel{\phi_i}{\cong} H)$ (Theorem \ref{main-coro}, and Remark \ref{core-remarks}(1)).

Equipped with Theorem \ref{intro-main-theorem}, one still faces the obstacle of verifying that the necessary and sufficient conditions of Bludov and Glass hold for a particular family of left-ordered groups (see the discussion preceding Theorem \ref{bludov-glass}), which is quite difficult in general.  However there are a few natural special cases where the conditions are somewhat easier to verify, leading to simplified versions of our main result.  For example, it is easy to amalgamate circularly-ordered groups along convex subgroups (as is also the case with left-ordered groups \cite[Corollary 5.2]{BV09}):

\begin{restatable}{proposition}{linamalg} \label{linear-amalgamation}
Suppose $(G_i,c_i)$ are circularly-ordered groups for $i \in I$, each equipped with a convex subgroup $H_i \subset G_i$ and an order-preserving isomorphism $\phi_i:(H,d) \to (H_i,c_i)$ from a circularly-ordered group $(H,d)$.  Then the group $*_{i \in I} G_i  (H_i \stackrel{\phi_i}{\cong} H)$ admits a circular ordering $c$ which extends the orderings $c_i$ of $G_i$ for $i \in I$.
\end{restatable}


This simplified version is of interest to us as it yields immediate applications to fundamental groups of $3$-manifolds, see Example \ref{convex-groups-and-3-manifolds} and the motivation below. 

It is also possible to amalgamate circularly-ordered groups along order-isomorphic rank-one abelian subgroups and their circularly-ordered analogues.  While useful (see, \cite[Lemma 4.12]{BCGpreprint}), this result is fundamentally different than the analogous result for left-orderable groups, which states that amalgamation of left-orderable groups along rank one abelian subgroups always produces a left-orderable group (\cite[Corollary 5.2]{BV09}, cf. Example \ref{CO-cyclic-amalgamation-failure}).  Nevertheless, once circular orderings of $3$-manifold groups are better understood, we expect this special case will also have applications to in the realm of $3$-manifold topology analogous to \cite[Theorem 2.7]{CLW13}.

\begin{restatable}{proposition}{cycamalg}\label{cyclic-amalgamation}
Suppose $(G_i,c_i)$ are circularly-ordered groups for $i \in I$, each equipped with a subgroup $H_i \subset G_i$ and an order-preserving isomorphism $\phi_i:(H,d) \to (H_i,c_i)$ from a circularly-ordered group $(H,d)$. 

 If $H$ is either:
 \begin{enumerate}
  \item a subgroup of the rational points of $S^1$ equipped with the standard ordering, or 
  \item $\mathbb{Q}$ or $\ZZ$ equipped with the ordering $d(q_1, q_2, q_3) =1$ if and only if $q_1<q_2<q_3$ (up to cyclic permutation),
  \end{enumerate}
  then $*_{i \in I} G_i  (H_i \stackrel{\phi_i}{\cong} H)$ admits a circular ordering that extends each of the $c_i$.
\end{restatable}

%

As implied above, our interest in extending the work of Bludov and Glass to the case of circularly-ordered groups stems from current work in low-dimensional topology.  Left-orderable groups have recently come to prominence in the field of low-dimensional topology via conjectured connections between foliations, Heegaard-Floer homology and left-orderability of the fundamental groups of $3$-manifolds \cite[Conjecture 1]{BGW13}, \cite[Conjecture 2.5]{Juhasz15}.  In this setting, the question of left-orderability of free products with amalgamation arises naturally when considering the fundamental groups of compact, orientable, non-geometric $3$-manifolds, as such fundamental groups are encoded by graphs of groups with edge groups isomorphic to $\ZZ \oplus \ZZ$.  

In some sense, however, circular orderability is the natural way to begin a study of left-orderability in the context of $3$-manifold fundamental groups, as there are often direct connections between circular orderings and the topology of the underlying manifold.  For example, if $M$ is a compact, connected, orientable $3$-manifold,  the existence of a circular ordering of $\pi_1(M)$ is tied directly to whether or not $M$ supports a co-orientable taut foliation, via Thurston's universal circle construction \cite{CD03}.  One can also create circular orderings of $\pi_1(M)$ when $M$ is a cusped hyperbolic $3$-manifold admitting a certain type of nice triangulation via a study of the action of $\pi_1(M)$ on the cusps of the universal cover of $M$ (\cite{SSpreprint}, cf.  Example \ref{convex-groups-and-3-manifolds}).  In contrast, aside from cases where $|H_1(M)| = \infty$ (in which case, $\pi_1(M)$ is known to be left-orderable by \cite{BRW05}), most left orderings of $3$-manifold fundamental groups arise by first constructing a circular ordering via a representation of $\pi_1(M)$ into a group of homeomorphisms of $S^1$, and then using one of a variety of ad-hoc techniques to show that the Euler class of the representation is trivial.

This work is therefore inspired by the following problem, raised in \cite[Section 4]{BS15}:\footnote{The authors of \cite{BS15} pose the question in terms of a decomposition of $\pi_1(M)$ arising from a Heegaard splitting.} 

\begin{problem}
\label{3manifold-problem}
 Suppose that $M$ is a $3$-manifold with geometric pieces $M_1, \ldots, M_n$, and that $\pi_1(M_i)$ admits a circular ordering $c_i$.  Determine necessary and sufficient conditions in terms of the gluing maps which recover $M$ from the pieces $M_i$, and the circular orderings $c_i$,  which guarantee the existence of an ordering $c$ of $\pi_1(M)$ extending each of the $c_i$ (cf. \cite[Theorem 1.7(2)]{BC17}).
\end{problem}

The paper is organized as follows.  In Section \ref{background} we review definitions and place them in a categorical framework.  In Section \ref{amalgamation} we expand the categories introduced in Section \ref{background} to include certain colimits and establish our main theorem.  We restate the theorem in the language of circular orderings in Section \ref{compatible-section}, obtaining a circularly-orderable analogue to the theorem of Bludov and Glass for amalgamations of left-orderable groups.  Section \ref{special-cases} covers two relevant special cases where amalgamation always yields a circularly-ordered group.  Last, Section \ref{tensor-cat} shows that the categories introduced in Section \ref{background} are in fact tensor categories when paired with an operation introduced in  \cite{BS15}.

\subsection*{Acknowledgements}
We would like to thank the referee for several suggestions that have improved the paper. Adam Clay was partially supported by NSERC grant RGPIN-2014-05465. Tyrone Ghaswala was partially supported by a PIMS Postdoctoral Fellowship at the University of Manitoba.

\section{Background and categorical framework}
\label{background}

We begin with the definition of a left ordering of a group, and what is commonly called the \textit{cocycle definition} of a circular ordering of a group.

\begin{definition}
A \emph{left ordering} of a group $G$ is a strict total ordering $<$ such that $g < h$ implies $fg < fh$ for all $f,g,h \in G$. When $G$ admits a left ordering $<$, we call $G$ \emph{left-ordered} and write $(G,<)$. Given a left ordering, we can define the \emph{positive cone} $P =\{ g \in G: g > id\}$.

Given left-ordered groups $(G,<)$ and $(H,\prec)$, an \emph{order-preserving homomorphism} is a homomorphism $\phi:G \to H$ such that $g_1 < g_2$ if and only if $\phi(g_1) \prec \phi(g_2)$ for all $g_1,g_2 \in G$.
\end{definition}

It is easily checked that a positive cone $P$ of a left ordering satisfies $P \cdot P \subset P$ and $P \sqcup P^{-1} = G \sm\{id\}$. On the other hand, given a subset $P$ of a group $G$ satisfying $P \cdot P \subset P$ and $P \sqcup P^{-1} = G \sm\{id\}$, we can define a left ordering $<$ with positive cone $P$ by $g < h$ whenever $g^{-1}h \in P$. Therefore to define a left ordering of a group it suffices to specify its positive cone.

\begin{definition} \label{cocycle-def}
Given a $G$-set $S$, an \emph{invariant circular ordering} on $S$ is a function $c:G^3 \to \{\pm 1,0\}$ such that
\begin{enumerate}
\item $c^{-1}(0) = \Delta(S)$, where $\Delta(S) := \{(a_1,a_2,a_3) \in S^3 \mid a_i = a_j, \text{ for some } i \neq j\}$,
\item the function $c$ satisfies the cocycle condition
\[
c(a_2,a_3,a_4)-c(a_1,a_3,a_4)+c(a_1,a_2,a_4)-c(a_1,a_2,a_3)=0
\]
for all $a_1,a_2,a_3,a_4 \in S$, and
\item $c(a_1,a_2,a_3) = c(g\cdot a_1,g\cdot a_2,g\cdot a_3)$ for all $g \in G$ and $a_1,a_2,a_3 \in S$.
\end{enumerate}
When $G$ admits an invariant circular ordering $c$ under the action of left multiplication on itself, we call $G$ \emph{circularly-ordered} and write $(G,c)$.

Given circularly-ordered groups $(G,c)$ and $(H,d)$ an \emph{order-preserving homomorphism} is a homomorphism $\phi:G \to H$ such that $c(g_1,g_2,g_3) = d(\phi(g_1),\phi(g_2),\phi(g_3))$ for all $g_1,g_2,g_3 \in G$.
\end{definition}

Note that order-preserving homomorphisms of both left-ordered groups and circularly ordered groups are necessarily injective. 

Let $\phi:G \to H$ be an injective homomorphism. Suppose $<$ is a left order on $H$ with positive cone $P$.  Define the {\it pullback of $<$ by $\phi$} as the left order $<^\phi$ on $G$ given by $g_1 <^\phi g_2$ if $\phi(g_1) < \phi(g_2)$.   The positive cone of $<^\phi$ is given by $\phi^*P := \{g \in G: \phi(g) \in P\}$.  Similarly, suppose $c$ is a circular ordering on $H$. Define the {\it pullback of $c$ by $\phi$} as the circular ordering $\phi^*c$ on $G$ given by $\phi^*c(g_1,g_2,g_3) := c(\phi(g_1),\phi(g_2),\phi(g_3))$. With this notation, an injective homomorphism $\phi:(G,c) \to (H,d)$ of circularly-ordered groups (resp. $\varphi:(G,<) \to (H,\prec)$ of left-ordered groups) is order-preserving exactly when $\phi^*d = c$ (resp. $\prec^\varphi = <$).

Important in the study of circularly-ordered groups is the relation between a group $(G, c)$ and its left-ordered lift, $(\widetilde{G}_c, <_c, z_c)$.  Here, and in what follows, the notation $(G, <, z)$ will be used to denote a left-ordered group $G$ with ordering $<$ and a chosen positive, cofinal, central element $z \in G$.  Recall an element $z \in G$ is \emph{cofinal} with respect to a left ordering $<$ of $G$ (or $<$-cofinal for short) if $$G= \{ g \in G \mid \exists k \in \ZZ \mbox{ such that } z^{-k} < g< z^k \}.$$  

\begin{construction}[\cite{zheleva76}]
\label{lift} 
Given a circularly-ordered group $(G,c)$, construct $(\widetilde{G}_c, <_c, z_c)$ as follows.  Let $\widetilde{G}_c$ denote the central extension of $G$ by $\ZZ$ constructed by equipping the set $\ZZ \times G$ with the operation $(n,a)(m,b)=(n+m+f_c(a,b),ab)$, where
\[
f_c(a,b)=\left\{\begin{array}{cl} 0 & \text{if $a=id$ or $b=id$ or $c( id, a, ab)=1$}
\\ 1 &  \text{if $ab=id$ $(a\not=id)$ or $c(id,ab,a)=1$.  }  \end{array}
\right.
\] 
Define the positive cone of a left ordering $<_c$ by $$P=\{(n,a)\mid n\geq 0\} \setminus \{ (0, id) \}.$$  The central element $z_c = (1, id)$ is positive and cofinal with respect to $<_c$.  When no confusion will arise from doing so, we will denote $\widetilde{G}_c$ by $\wt G$.
\end{construction}

It can be checked that if $\phi:H \to G$ is an injective homomorphism and $c$ is a circular ordering of $G$, then $f_{\phi^*c} = \phi^*f_c$ where $\phi^*f_c(h_1,h_2) := f_c(\phi(h_1),\phi(h_2))$ for all $h_1,h_2 \in H$.

\begin{remark}\label{derived-cocycle}
The functions $c:G^3 \to \ZZ$ and $f_c:G^2 \to \ZZ$ are both 2-cocycles, where $c$ is expressed in homogeneous coordinates, and $f$ is expressed in inhomogeneous coordinates.  In fact, $[c] = 2[f_c]$ in $H^2(G;\ZZ)$ and Construction \ref{lift} is the well-known construction that gives rise to a bijection between elements of $H^2(G;\ZZ)$ and equivalence classes of central extensions of $G$ \cite[Chapter IV.3]{Brown}.  Indeed, consider the set-theoretic section $s:G \to \wt G_c$ of the central extension
\[
1 \lra \langle z_c \rangle \overset{\iota}{\lra} \wt G_c \lra G \to 1
\]
by defining $s(g) \in \wt G_c$ to be the unique element such that $id \leq_c s(g) <_c z_c$.  Then
\[
\iota \left(z_c^{f_c(a,b)}\right) = s(a)s(b)s(ab)^{-1}
\]
(see Lemma \ref{key-equivalence}), that is, $f_c(a,b)$ measures the failure of $s$ to be a homomorphism.  In other words, $[f_c] \in H^2(G;\ZZ)$ is the Euler class of the identity homomorphism $G \to G$.  Note that it is possible for two different circular orderings $c,d$ on $G$ to be such that $[f_c] = [f_d] \in H^2(G;\ZZ)$.  While this implies that the central extensions $\wt G_c$ and $\wt G_d$ are isomorphic, it may be that $\wt G_c$ and $\wt G_d$ are not isomorphic as left-ordered groups.  
\end{remark}

When $(G, <, z)$ is a left-ordered group with a positive cofinal central element $z$, we can take a quotient of $G$ by $\langle z \rangle$ and arrive at a circularly-ordered group.  
\begin{construction}[\cite{zheleva76}]\label{quotient}
Given $(G, <, z)$, let $\ol G = G/\langle z \rangle$.  Define a circular ordering $c_<$ on $\ol{G}$ as follows: for every $g\langle z \rangle \in G/\langle z \rangle$, define the {\em minimal representative} of $g\langle z \rangle$ to be the unique $\ol{g} \in g\langle z \rangle$ satisfying $id \leq \ol{g} <z$. Then set 
$$c_<(g_1\langle z \rangle, g_2\langle z \rangle,g_3\langle z \rangle)=sign(\sigma),$$
where $\sigma$ is the unique permutation in $S_3$ such that $\ol{g_{\sigma(1)}}<\ol{g_{\sigma(2)}}<\ol{g_{\sigma(3)}}$. 
\end{construction}

These two constructions are not inverses to one another, but provide an equivalence of categories in a sense that we now make precise.

\begin{definition}
\label{the-little-categories}
Define a category $\mathbf{Circ}$ whose objects are circularly-ordered groups $(G,c)$, and whose morphisms $\phi:(G,c) \to (H,d)$ are order-preserving homomorphisms.

Define a category $\LOs$ whose objects are left-ordered groups $(G,<,z)$ equipped with a central, positive, cofinal element.  Morphisms $\phi:(G,<,z) \to (H,\prec,w)$ are order-preserving homomorphisms $\phi:G \to H$ such that $\phi(z) = w$.
\end{definition}

It is tempting to define categories whose morphisms include non-injective homomorphisms, say by replacing the condition $c(g_1, g_2, g_3) = d(\phi(g_1), \phi(g_2), \phi(g_3))$ on morphisms in $\Circ$ with $$|c(g_1, g_2, g_3) - d(\phi(g_1), \phi(g_2), \phi(g_3))| \leq 1$$ and similarly modifying the definition of morphisms in $\LOs$ (this would allow quotient maps where the kernel is a convex subgroup, c.f. Lemma \ref{convex}).  However, with this modification $\Circ$ and $\LOs$ are no longer equivalent categories, and the construction in Section \ref{tensor-cat} does not yield a bifunctor on $\Circ$.  See Remark \ref{importance-of-injectivity} and Proposition \ref{functor-theorem} for more details on this point.

We now build functors $L:\Circ \to \LOs$ and $Q:\LOs \to \Circ$ in the following way: on objects, define $L$ and $Q$ by Constructions \ref{lift} and \ref{quotient} respectively.  If $\phi:(G,c) \to (H,d)$ is a morphism in $\Circ$, then define $L(\phi) = \wt \phi:\wt G \to \wt H$ by $\wt\phi((n,a)) = (n,\phi(a))$.  For a morphism $\theta:(G,<,z) \to (H,\prec,w)$ in $\LOs$, define $Q(\theta) = \ol\theta:\ol G \to \ol H$ by $\ol\theta(g\langle z\rangle) = \theta(g)\langle w \rangle$.  The proof of the following lemma is a straightforward calculation, so we omit it. 

%

\begin{lemma}\label{functors}
The mappings $L:\Circ \to \LOs$ and $Q:\LOs \to \Circ$ are well-defined functors.
\end{lemma}

To prove that $L$ and $Q$ give an equivalence of categories, we must first prove the following key technical lemma.  Given $(G,<,z)$, define the function $f_<:\ol G^2 \to \ZZ$ by 
\[
z^{f_<(a\langle z \rangle,b\langle z \rangle)} = (\ol a)(\ol b)(\ol{ab})^{-1}.
\]
Given $(G,c)$, let $(\wt G,<_c,z_c)$ be the object obtained by applying the functor $L$.  Define an isomorphism $\eta_G:\ol{\wt G} \to G$ by $(n,a)\langle z_c \rangle \mapsto a$.  In the next proof, note that for an extension built from a 2-cocycle $f$ as in Construction \ref{lift}, $(n,a)^{-1} = (-n-f(a,a^{-1}),a^{-1}) = (-n-1,a^{-1})$.  

\begin{lemma}\label{key-equivalence}
With the notation above, $f_< = f_{c_<}$ and $\eta_G^*c = c_{<_c}$.
\end{lemma}
\begin{proof}
We'll first show $f_< = f_{c_<}$.  If $\ol a = id$ or $\ol b = id$, then we immediately have $$f_<(a\langle z \rangle,b\langle z \rangle) = f_{c_<}(a\langle z \rangle,b\langle z \rangle) = 0.$$ Since $z$ is central in $G$, we have $id \leq \ol a \ol b < z^2$.  If $\ol a \ol b < z$, then $\ol{ab} = \ol a \ol b$ so $f_<(a \langle z \rangle, b\langle z \rangle) = 0$.  Since $id < \ol b$, $id < \ol a < \ol a \ol b = \ol{ab}$ so $c_<(id,a\langle z \rangle, ab\langle z \rangle) = 1$ and $f_{c_<}(a \langle z \rangle, b\langle z \rangle) = 0$.  On the other hand, assume $z \leq \ol a \ol b$. Then $\ol{ab}z= \ol a \ol b$ so $\ol a \ol b (\ol{ab})^{-1} = z$ and $f_<(a \langle z \rangle, b\langle z \rangle) = 1$.  Since $\ol b < z$, $\ol{ab} = z^{-1}\ol a\ol b < \ol a$.  Therefore $id\leq \ol{ab} < \ol a$.  If $id = \ol{ab}$, then $a\langle z \rangle b\langle z \rangle = id$ so we have $f_{c_<}(a \langle z \rangle, b \langle z \rangle) = 1$.  If $id < \ol{ab}$, then $c_<(id, ab\langle z \rangle, a\langle z \rangle) = 1$ implying $f_{c_<}(a \langle z \rangle, b\langle z \rangle) = 1$ and we may conclude $f_< = f_{c_<}$.

We now show $\eta_G^*c = c_{<_c}$.  Minimal representatives in $\wt G$ take the form $\ol{(n,a)} = (0,a)$.  Notice that $(0,a) < (0,b)$ if and only if $(0,a)^{-1}(0,b) = (f_c(a^{-1},b)-1,a^{-1}b)$ is in the positive cone of $<_c$, which occurs precisely when $f_c(a^{-1},b) = 1$.  Consider an arbitrary triple $((n_1,a_1)\langle z_c\rangle, (n_2,a_2) \langle z_c\rangle, (n_3,a_3) \langle z_c \rangle) \in \ol G \sm \Delta(\ol G)$.  Let $\sigma \in S_3$ be the unique permutation such that $(0,a_{\sigma(1)}) <_c (0,a_{\sigma(2)}) <_c (0,a_{\sigma(3)})$, which is equivalent to
\[
f_c(a_{\sigma(1)}^{-1},a_{\sigma(2)}) = f_c(a_{\sigma(2)}^{-1},a_{\sigma(3)}) = 1.
\]
Since $(a_{\sigma(1)},a_{\sigma(2)},a_{\sigma(3)})\notin \Delta(G)$, this is equivalent to the condition
\[
c(id,a_{\sigma(1)}^{-1}a_{\sigma(2)},a_{\sigma(1)}^{-1}) = c(id,a_{\sigma(2)}^{-1}a_{\sigma(3)},a_{\sigma(2)}^{-1}) = 1.
\]
Since $c$ is invariant under left multiplication we have $c(a_{\sigma(1)},a_{\sigma(2)},id) = c(a_{\sigma(2)},a_{\sigma(3)},id) = 1$.  Applying the cocycle condition gives
$c(a_{\sigma(1)},a_{\sigma(2)},a_{\sigma(3)}) = 1$.  Therefore we have
\[
c_{<_c}((n_1,a_1)\langle z_c\rangle, (n_2,a_2) \langle z_c\rangle, (n_3,a_3) \langle z_c \rangle) = sign(\sigma) = c(a_1,a_2,a_3),
\]
completing the proof since $\eta_G((n_i,a_i)\langle z_c\rangle) = a_i$.
\end{proof}

\begin{proposition}\label{categorical-equivalence}
The functors $Q$ and $L$ provide an equivalence of categories $\mathbf{LO}_* \cong \mathbf{Circ}$.
\end{proposition}
\begin{proof}
First note that isomorphisms in $\Circ$ and $\LOs$ are simply morphisms that are also group isomorphisms, since both categories admit a faithful functor to the category of groups.  We will start by showing $LQ \simeq \mathbf 1_{\LOs}$.  Let $(G,<,z)$ be an object in $\LOs$ and note that every element of $G$ can be written uniquely as $z^n \ol a$, where $n \in \ZZ$ and $\ol a \in a \langle z \rangle$ is the minimal representative.  Viewed this way, the group structure is given by $(z^n \ol a)(z^m \ol b) = z^{n+m+f_<(a\langle z \rangle, b\langle z \rangle)}\ol{ab}$ where $f_<$ is the cocycle from Lemma \ref{key-equivalence}.  Furthermore, $id < z^n\ol a$ if and only if $z^{-n}< \ol a$, which occurs precisely when $n \geq 0$.  Therefore the positive cone of $<$ is given by
\[
P_< = \{z^n \ol a \in G \mid n \geq 0\}\sm\{id\}.
\]
Now construct a map $\nu_G:LQ(G) \to G$ by $(n,a\langle z \rangle) \mapsto z^n\ol a$.  This map is a bijection, and since $f_< = f_{c_<}$ (by Lemma \ref{key-equivalence}) it is a group isomorphism.  Since $\nu_G((1,id)) = z$ and since $(n,a\langle z \rangle)$ in the positive cone of $<_{c_<}$ is mapped to $\nu_G(n,a\langle z \rangle) = z^n\ol a \in P_<$, we conclude $\nu_G$ is an isomorphism in $\LOs$.  It remains to check the $\nu_G$ give natural isomorphisms between $LQ$ and $\mathbf 1_{\LOs}$.  Let $\theta:(G,<,z) \to (H,\prec,w)$ be a morphism in $\LOs$.  Then
\[
\nu_H\wt{\ol{\theta}}((n,a\langle z \rangle)) = \nu_H((n,\theta(a)\langle w \rangle)) = w^n\ol{\theta(a)} = w^n\theta(\ol a) = \theta(z^n\ol a) = \theta\nu_G((n,a\langle z \rangle))
\]
so $LQ \simeq \mathbf 1_{\LOs}$.

To see $QL \simeq \mathbf 1_{\Circ}$, recall the maps $\eta_G:QL(G) \to G$ from Lemma \ref{key-equivalence}.  These are easily checked to be group isomorphisms, and since $\eta_G^*c = c_{<_c}$ by Lemma \ref{key-equivalence}, the $\eta_G$ are isomorphisms in $\Circ$.  Let $\phi:(G,c) \to (H,d)$ be a morphism in $\Circ$.  Then
\[
\eta_H\ol{\wt\phi}((n,a)\langle z_c\rangle) = ((n,\phi(a))\langle z_d\rangle) = \phi(a) = \phi\eta_G((n,a)\langle z_c \rangle)
\]
so the $\eta_H$ give a natural isomorphism $QL \simeq \mathbf 1_{\Circ}$.  We conclude $\LOs \cong \Circ$.
\end{proof}

\begin{remark}
\label{sketch-of-proof}
With these notions established, it is possible to give a rough sketch of the ideas that follow.  Suppose that $(G_i, c_i)$ are circularly-ordered groups for $i \in I$ with subgroups $H_i$ and order-preserving isomorphisms $\phi_i:(H, d) \rightarrow (H, c_i)$ from some circularly-ordered group $(H,d)$.  The categorical equivalence outlined above yields, for each $i$, an isomorphism $\psi_i : \ol{\wt G}_i \rightarrow G_i$.  These isomorphisms piece together to yield a map $\psi : *_{i \in I} \wt G_i(\wt H_i \stackrel{\wt \phi_i}{\cong}\wt{H})/ \langle z \rangle  \rightarrow *_{i\in I}G_i(H_i \stackrel{\phi_i}{\cong} H)$ when we identify each $\ol{\wt{G}}_i$ with the subgroup $\wt G_i /\langle z \rangle$ in the quotient $*_{i \in I} \wt G_i(\wt H_i \stackrel{\wt \phi_i}{\cong}\wt{H})/ \langle z \rangle$ (here, $z$ is the cofinal central element in the free product with amalgamation that results from identifying all of the cofinal central elements of the factors).  

Assuming that it is possible to extend the left orderings $<_{c_i}$ of the lifts to a left ordering of the group $ *_{i \in I} \wt G_i(\wt H_i \stackrel{\wt \phi_i}{\cong}\wt{H})/ \langle z \rangle$, one checks that $z$ is necessarily cofinal in the resulting ordering so that the group $*_{i\in I}G_i(H_i \stackrel{\phi_i}{\cong} H)$ inherits a circular ordering by applying Construction \ref{quotient}.  A similar argument proves the other direction of Theorem \ref{intro-main-theorem} (see the proof of Theorem \ref{main-coro}).  Note that the isomorphisms one constructs in each case appear inherently categorical in nature---something that we explain in the next section.
\end{remark}

\begin{remark}\label{importance-of-injectivity}
Suppose one were to modify the definitions of $\Circ$ and $\LOs$ to allow for non-injective homomorphisms, as in the comments following Definition \ref{the-little-categories}.  While the categories themselves will still be well-defined, Constructions \ref{lift} and \ref{quotient} can no longer be defined on morphisms in a way that yields an equivalence of categories (despite the fact that one obtains a bijection on the objects up to isomorphism), and so the corresponding generalization of Proposition \ref{categorical-equivalence} fails.  To see this, note that the object $(\{1\},c_1)$ in the modified category of circularly-ordered groups is the terminal object, while its lift $(\ZZ,<,1)$ is not terminal (for example, there is no morphism $(\QQ,<,1) \to (\ZZ,<,1)$).  Because of this, the arguments of Proposition \ref{categorical-equivalence} break down, as whenever $\phi:(G,c) \to (H,d)$ is not injective the condition $\phi^*f_d = f_c$ fails (a key fact in proving $L$ is a well-defined functor).  Interestingly, although $L$ loses its status as a well-defined functor, $Q$ remains a faithful functor that is bijective on objects up to isomorphism, but it is no longer full.
\end{remark}

\section{Circularly ordering free products with amalgamation}\label{amalgamation}

An \emph{amalgamation diagram} in a category is a diagram consisting of an object $A$, a set of objects $\{G_i\}_{i \in I}$ and for each $i \in I$, a morphism $\varphi_i:A \to G_i$.  We will denote such a diagram by $(A,\{(\varphi_i,G_i)\}_{i \in I})$.

With the goal of circularly ordering free products with amalgamation of circularly-ordered groups in a way compatible with the ordering of each factor, one may hope to simply investigate colimits of amalgamation diagrams in $\Circ$.  However this approach cannot possibly work, since even when the corresponding free product with amalgamation of the underlying groups is circularly-orderable (with an order extending the orders on the factors), there is no corresponding colimit in $\Circ$.  For example, consider the amalgamation diagram
\[
\vD = ((\{id\},c_0), \{(\iota_1,(A_1,d_1)), (\iota_2,(A_2,d_2))\})
\]
in $\Circ$ where $A_i = \ZZ$ and $d_i = c$ is the circular ordering on $\ZZ$ determined by $c(x,y,z) = 1$ whenever $x < y < z$.  Consider the free group $F_2 = \langle a_1,a_2 \rangle$ and identify $A_1$ and $A_2$ via inclusion $A_i = \langle a_i \rangle < F_2$.  Since there is a circular ordering on $F_2$ extending the circular orderings on $A_1$ and $A_2$ \cite{BS15}, if a colimit for $\vD$ exists it must be of the form $(F_2,d)$ for some circular ordering $d$ on $F_2$.  

Now consider the identity morphisms $(A_i,d_i) \to (\ZZ,c)$.  If $(F_2,d)$ were indeed a colimit of the diagram, by the universal property there would be a morphism $(F_2,d) \to (\ZZ,c)$.  However, $\Circ$ does not admit non-injective homomorphisms, so a colimit of $\vD$ cannot exist.

Even though taking colimits in $\Circ$ is impossible, it is still possible to obtain a circular ordering on a free products with amalgamation (compatible with given orderings of the factors) via a categorical construction.  To do this, we will embed $\Circ$ and $\LOs$ in categories $\BigC$ and $\BigL$ respectively that do admit colimits corresponding to circularly-ordered and left-ordered free products with amalgamation.

Many of the building blocks for these categories are familiar constructions that can be found in any elementary group cohomology textbook, such as \cite{Brown}.

\subsection{The big categories}

A \emph{sectioned central extension} is the data $(E,G,\iota,\pi,s)$ where
\begin{center}
\begin{tikzcd}
	1  \arrow[r] & \ZZ \arrow[r,"\iota"] & E \arrow[r,"\pi"] & G \arrow[l,"s",bend left = 20] \arrow[r] &1
\end{tikzcd}
\end{center}
is a central extension with a set-theoretic section $s:G \to E$ such that $s(id) = id$.  A \emph{sectioned central extension morphism} $\theta:(E,G,\iota,\pi,s) \to (F,H,\epsilon,\rho,t)$ is a group homomorphism $\theta:E \to F$ such that $\theta\iota = \epsilon$ and $\theta s = t \ol \theta$.  Here $\ol\theta:G \to H$ is defined by $\ol\theta(\pi(g)) = \rho\theta(g)$. A sectioned central extension morphism that is also a group isomorphism is called a \emph{sectioned central extension isomorphism}.

We say sectioned central extensions $(E,G,\iota,\pi,s)$ and $(F,G,\epsilon,\rho,t)$ are \emph{equivalent} if there exists a sectioned central extension morphism $\theta:E \to F$ such that $\ol\theta$ is the identity map $G \to G$.  Such a morphism is called an \emph{equivalence}.  Note that all equivalences are sectioned central extension isomorphisms, but the converse does not hold.

Recall that for a group $G$, a \emph{normalized 2-cocycle} is a function $f:G^2 \to \ZZ$ such that $f(id,g) = f(g,id) = 0$ for all $g \in G$ and $f(g_2,g_3) - f(g_1g_2,g_3) + f(g_1,g_2g_3) - f(g_1,g_2) = 0$ for all $g_1, g_2,g_3 \in G$.  Denote the set of such cocycles by $\Gamma^2(G,\ZZ)$.

\begin{definition} \label{big-categories}
Define the category $\mathbf{BigCirc}$ as follows.  Objects are pairs $(G,S)$ where $G$ is a group and $S \subset \Gamma^2(G,\ZZ)$ is a non-empty subset.  A morphism $\phi:(G,S) \to (H,T)$ is a group homomorphism $\phi:G \to H$ such that $\phi^*(T) \subset S$.  

Define the category $\mathbf{BigLO_*}$ as follows.  Objects are non-empty sets $\{(E_\alpha,G,\iota_\alpha,\pi_\alpha,s_\alpha)\}_{\alpha \in \vA}$ of sectioned central extensions such that no two in the set are equivalent.  A morphism 
\[
\theta_\vB: \{(E_\alpha,G,\iota_\alpha,\pi_\alpha,s_{\alpha})\}_{\alpha \in \vA} \to \{(F_\beta,H,\epsilon_\beta,\rho_\beta,t_\beta)\}_{\beta \in \vB}
\]
is a set of sectioned central extension morphisms
\[
\theta_\vB = \{\theta_\beta:E_{\alpha_\beta} \to F_\beta \mid \ol\theta_\beta = \ol\theta_{\beta'} \text{ for all } \beta,\beta' \in \vB\}.
\]
The identity morphism is given by $\theta_\alpha = \id:E_\alpha \to E_\alpha$ for all $\alpha \in \vA$.  Given 
\[
\{(E_\alpha,G,\iota_\alpha,\pi_\alpha,s_\alpha)\}_{\alpha \in \vA} \overset{\theta_\vB}{\lra} \{(F_\beta,H,\epsilon_\beta,\rho_\beta,t_\beta)\}_{\beta \in \vB} \overset{\psi_\Lambda}{\lra} \{(I_\lambda,J,\nu_\lambda,\omega_\lambda,r_\lambda)\}_{\lambda \in \Lambda},
\]
define $\psi_\Lambda\theta_\vB:\{(E_\alpha,G,\iota_\alpha,\pi_\alpha,s_\alpha)\}_{\alpha \in \vA}  \to  \{(I_\lambda,J,\nu_\lambda,\omega_\lambda,r_\lambda)\}_{\lambda \in \Lambda}$ by the set of sectioned central extension morphisms $\{\psi_\lambda\theta_\beta \mid \beta = \beta_\lambda\}$.  
\end{definition}
The definition of composition of morphisms in $\BigL$ can be rephrased in plain language by saying that we create the set of all possible compositions of sectioned central extension morphisms.

Recall the following standard constructions from \cite[Chapter IV.3]{Brown}. Given $f \in \Gamma^2(G,\ZZ)$, construct the \emph{associated sectioned central extension} $(\wt G_f,G,\iota_f,\pi_f,s_f)$ as follows.  Let $\wt G_f$ be the group with underlying set $\ZZ \times G$ and multiplication defined by $(a,g)(b,h) =(a + b + f(g,h),gh)$.  Define $\iota_f(a) = (a,1)$, $\pi_f((a,g)) = g$, and $s_f(g) = (0,g)$.  Conversely, given a sectioned central extension $(E,G,\iota,\pi,s)$ define the \emph{associated cocycle} $f_s \in \Gamma^2(G,\ZZ)$ by $\iota f_s(g,h) = s(g)s(h)s(gh)^{-1}$.

Define the map $\fL:\mathbf{BigCirc} \to \mathbf{BigLO_*}$ as follows.  On objects, map $(G,\{f_\alpha\}_{\alpha \in \vA})$ to the set of associated sectioned central extensions $\{(\wt G_{f_\alpha},G,\iota_{f_\alpha},\pi_{f_\alpha},s_{f_\alpha})\}_{\alpha \in\vA}$.  Let
\[
\phi:(G,\{f_\alpha\}_{\alpha \in \vA}) \to (H,\{f_\beta\}_{\beta \in \vB})
\]
be a morphism in $\mathbf{BigCirc}$.  For each $\beta \in \vB$, there is some $\alpha_\beta \in \vA$ such that $\phi^*f_\beta = f_{\alpha_\beta}$.  Define the sectioned central extension morphism $\wt\phi_\beta:\wt G_{f_{\alpha_\beta}} \to \wt H_{f_\beta}$ by $\wt\phi_\beta(a,g) = (a,\phi(g))$.  Under the functor $\fL$, map $\phi$ to $\{\wt\phi_\beta \mid \beta \in \vB\}$.

On the other hand, define the map $\fQ:\mathbf{BigLO_*} \to \mathbf{BigCirc}$ as follows.  Map an object \linebreak $\{(E_\alpha,G,\iota_\alpha,\pi_\alpha,s_\alpha)\}_{\alpha \in \vA}$ in $\mathbf{BigLO_*}$ to the object $(G,\{f_{s_\alpha}\}_{\alpha \in \vA})$ in $\mathbf{BigCirc}$.  Map a morphism \linebreak $\{\theta_\beta\}_{\beta \in \vB}:\{(E_\alpha,G,\iota_\alpha,\pi_\alpha,s_{\alpha})\}_{\alpha \in \vA} \to \{(F_\beta,H,\epsilon_\beta,\rho_\beta,t_\beta)\}_{\beta \in \vB}$ to the morphism $\ol\theta_\beta:G \to H$.

We wish to show that these rules for $\fL$ and $\fQ$ define functors that give an equivalence of categories.  Furthermore, we will see that $\Circ$ and $\LOs$ naturally embed in $\BigC$ and $\BigL$ in such a way that $L$ and $Q$ are restrictions of $\fL$ and $\fQ$.

\begin{lemma} \label{scext-equivalence}
Sectioned central extensions $(E_1,G,\iota_1,\pi_1,s_1)$ and $(E_2,G,\iota_2,\pi_2,s_2)$ are equivalent if and only if the associated cocycles $f_1,f_2 \in \Gamma^2(G,\ZZ)$ are equal.
\end{lemma}
\begin{proof}
Suppose $\theta:E_1 \to E_2$ is an equivalence.  Then 
\[
\iota_2f_2(a,b) = s_2(a)s_2(b)s_2(ab)^{-1} = \theta(s_1(a)s_1(b)s_1(ab)^{-1}) = \theta \iota_1f_1(a,b) = \iota_2f_1(a,b)
\]
and since $\iota_2$ is injective, $f_1 = f_2$.  Conversely, note that every element in $E_i$ can be uniquely written as $\iota_i(n)s_i(a)$ for some $n \in \ZZ$ and $a \in G$.  If $f_1 = f_2$, then the map $\theta:E_1 \to E_2$ given by \linebreak $\theta(\iota_1(n)s_1(a)) = \iota_2(n)s_2(a)$ is the desired equivalence of sectioned central extensions.
\end{proof} 

\begin{lemma} \label{big-functors}
The maps $\fL:\BigC \to \BigL$ and $\fQ:\BigL \to \BigC$ are functors.
\end{lemma}
\begin{proof}
The map $\fL$ is well-defined on objects by Lemma \ref{scext-equivalence}.  To prove $\fL$ is well-defined on morphisms, let $f_i = \Gamma^2(G_i,\ZZ)$ for $i = 1,2$, and $\phi:G_1 \to G_2$ be a homomorphism such that $\phi^*f_2 = f_1$.  It suffices to show $\wt \phi:((\wt G_1)_{f_1},G_1,\iota_{f_1},\pi_{f_1},s_{f_1}) \to ((\wt G_2)_{f_2},G_2,\iota_{f_2},\pi_{f_2},s_{f_2})$ given by $(n,a) = (n,\phi(a))$ is a sectioned central extension morphism.  We have $\wt\phi \iota_1 = \iota_2$ and $\wt \phi s_1 = s_2 \phi$ since $\ol{\wt\phi} = \phi$.  Since $\phi^*f_2 = f_1$, $\wt\phi$ is a homomorphism.  It is easy to check $\wt{\phi\varphi} = \wt\phi\wt\varphi$ and the identity map is lifted to the identity map.  Therefore $\fL$ is a well-defined functor.

To check $\fQ$ is a well-defined functor it suffices to verify that if
\[
\theta:(E,G,\iota,\pi,s) \to (F,H,\epsilon,\rho,t)
\]
is a sectioned central extension morphism, then $\ol \theta^*f_t = f_s$.  We have
\[
\epsilon f_t(\ol\theta(a),\ol\theta(b)) = t(\ol\theta(a))t(\ol\theta(b))t(\ol\theta(ab))^{-1} = \theta(s(a)s(b)s(ab)^{-1}) = \theta\iota f_s(a,b) = \epsilon f_s(a,b)
\]
and since $\epsilon$ is injective, $\ol \theta^*f_t = f_s$.  Noting that $\fQ$ preserves identity morphisms and $(\ol\theta)(\ol\psi) = \ol{\theta\psi}$ completes the proof.
\end{proof}


\begin{proposition}\label{big-categorical-equivalence}
The functors $\fQ$ and $\fL$ provide an equivalence of categories $\BigL \cong \BigC$.
\end{proposition}
\begin{proof}
We first show $\fQ\fL = \mathbf{1}_\BigC$.  Suppose $f \in \Gamma^2(G,\ZZ)$ and consider the associated sectioned central extension $(\wt G_f,G,\iota_f,\pi_f,s_f)$.  Then 
\[
\iota_f(f_{s_f}(a,b)) = (0,a)(0,b)(0,ab)^{-1} = (f(a,b),ab)(-f(ab,(ab)^{-1}),(ab)^{-1}) = (f(a,b),id)
\]
so $f = f_{s_f}$. It follows that $\fQ\fL(G,S) = (G,S)$ for all $(G,S) \in \BigC$.  Moreover for every morphism $\phi:(G,S) \to (H,T)$ we have  $\ol{\wt\phi} = \phi$, so we conclude $\fQ\fL = \mathbf{1}_\BigC$.

Next we will show $\fL\fQ \simeq \mathbf{1}_\BigL$.  Let $(E,G,\iota,\pi,s)$ be a sectioned central extension.  Every element of $E$ is uniquely written as $\iota(n)s(a)$ for $n \in \ZZ$ and $a \in G$.  Furthermore, $$(\iota(n)s(a))(\iota(m)s(b)) = \iota(n+m+f_s(a,b))s(ab).$$  Define $\mu_E:\wt G_{f_s} \to E$ by
\[
\mu_E((n,a)) = \iota(n)s(a)
\]
and observe $\mu_E$ is a sectioned central extension isomorphism.  Note that isomorphisms in $\BigL$ are sets of sectioned central extension isomorphisms.  Now suppose $\{(E_\alpha,G,\iota_\alpha,\pi_\alpha,s_\alpha)\}_{\alpha \in \vA}$ is an object in $\BigL$.  Abusing notation, let $f_\alpha = f_{s_\alpha}$ and define the isomorphism
\[
\mu_{E_\vA} = \{\mu_{E_\alpha}:\wt G_{f_\alpha} \to E_\alpha\}:\{(\wt G_{f_\alpha},G,\iota_{f_\alpha},\pi_{f_\alpha},s_{f_\alpha})\}_{\alpha \in \vA} \to \{(E_\alpha,G,\iota_\alpha,\pi_\alpha,s_\alpha)\}_{\alpha \in \vA}.
\]  
We will show the $\mu_{E_\vA}$ give a natural isomorphism $\fL\fQ \simeq \mathbf{1}_\BigL$.  Let 
\[
\theta_\vB = \{\theta_\beta:E_{\alpha_\beta} \to F_{\beta}\}:\{(E_\alpha,G,\iota_\alpha,\pi_\alpha,s_{\alpha})\}_{\alpha \in \vA} \to \{(F_\beta,H,\epsilon_\beta,\rho_\beta,t_\beta)\}_{\beta \in \vB}
\]
be a morphism in $\BigL$.  Fix $\beta \in \vB$.  Then 
\[
\theta_\beta\mu_{E_{\alpha_\beta}}((n,a)) = \theta_\beta(\iota_{\alpha_\beta}(n)s_{\alpha_\beta}(a)) = \epsilon_\beta(n)t_\beta(\ol\theta_\beta(a)) = \mu_{F_\beta}((n,\ol\theta_\beta(a))) = \mu_{F_\beta}\wt{\ol\theta}_\beta((n,a))
\]
so $\theta_\beta\mu_{E_{\alpha_\beta}} = \mu_{F_\beta}\wt{\ol\theta}_\beta:\wt G_{f_{\alpha_\beta}} \to F_\beta$.  Then
\[
\theta_\vB\mu_{E_\vA} = \{\theta_\beta\mu_{E_{\alpha_\beta}} \mid \beta \in \vB\} = \{\mu_{F_\beta}\wt{\ol\theta}_\beta \mid \beta \in \vB\} = \mu_{F_\vB}\fL\fQ(\theta_\vB).
\]
Therefore $\fL\fQ \simeq \mathbf{1}_\BigL$, completing the proof.
\end{proof}

We now shift our attention to identifying $\mathbf{Circ}$ and $\mathbf{LO_*}$ embedded in $\mathbf{BigCirc}$ and $\mathbf{BigLO_*}$ respectively.  Define a functor $\fI_{C}:\mathbf{Circ} \to \mathbf{BigCirc}$ by $\fI_C(G,c) = (G,\{f_c\})$, where $f_c \in \Gamma^2(G,\ZZ)$ is defined in Construction \ref{lift}.  On morphisms, set $\fI_C(\phi) = \phi$.  Since a morphism $\phi:(G,c) \to (H,d)$ in $\Circ$ is an injective homomorphism, $\phi^*f_d = f_c$ so $\fI_C$ is a well-defined functor.

Define a functor $\fI_L:\mathbf{LO_*} \to \mathbf{BigLO_*}$ as follows.  Let $(G,<,z)$ be an object in $\mathbf{LO_*}$.  Define \linebreak $\fI_L(G,<,z) = \{(G,G/\langle z \rangle,\iota,\pi,s)\}$ where $\iota(1) = z$, $\pi:G\to G/\langle z \rangle$ is the quotient map, and $s(g\langle z \rangle) = \ol g$ where $\ol g$ is the minimal representative of $g\langle z \rangle$.  On morphisms, define $\fI_L(\theta) = \{\theta\}$.  To see $\fI_L$ is a well-defined functor, let $\theta:(G,<,z) \to (H,\prec,w)$ be a morphism in $\LOs$, and let $\fI_L(G,<,z) = (G,G/\langle z \rangle,\iota,\pi,s)$ and $\fI_L(H,\prec,w) = (H,H /\langle w \rangle,\epsilon,\rho,t)$.  Since $\theta$ is a morphism in $\LOs$, $\theta(z) = w$ and $\theta(\ol g) = \ol{\theta(g)}$ for all $g \in G$.  Therefore $\theta\iota = \epsilon$ and $\theta s = t \ol \theta$ so $\fI_L$ is a well-defined functor.

\begin{lemma}\label{useful-inclusion-properties}
The functors $\fI_L$ and $\fI_C$ have the following properties.
\begin{enumerate}
\item The functors $\fI_L$ and $\fI_C$ are faithful.
\item A morphism $\phi:\fI_C(G,c) \to \fI_C(H,d)$ in $\BigC$ is of the form $\fI_C(\phi)$ if and only if $\phi$ is injective.
\item A morphism $\{\theta\}:\fI_L(G,<,z) \to \fI_L(H,\prec,w)$ is of the form $\fI_L(\theta)$ if and only if $\theta$ is injective.
\item The functors $\fI_L$ and $\fI_C$ are injective on objects.
\item The functors $\fI_L$ and $\fI_C$ are injective on morphisms.
\end{enumerate}
\end{lemma}
\begin{proof}
Property (1) is immediate.  For (2), since all morphisms in $\Circ$ are injective, any morphism of the form $\fI_C(\phi)$ in $\BigC$ is also injective.  Conversely, suppose $(G,c)$ and $(H,d)$ are circularly-ordered groups and $\phi:G \to H$ is injective such that $\phi^*f_d = f_c$.  Then by the proof of Proposition \ref{categorical-equivalence}, $\phi = \eta_H\ol{\wt\phi}\eta_G^{-1}$ so $\phi$ is a morphism in $\Circ$.  For (3), let $\fI_L(G,<,z) = \{(G,\ol G, \iota,\pi,s)\}$ and note that the positive cone $P_< \subset G$ of the left order $<$ is given by $\{\iota(n)s(a\langle z \rangle) \mid n \geq 0\} \sm\{id\}$.  Let $\fI_L(H,\prec,w) = \{(H, \ol H, \epsilon, \rho, t)\}$ and let $\theta:G \to H$ be an injective sectioned central extension morphism.  Then $\theta(\iota(n)s(a\langle z \rangle)) = \epsilon(n)t(\theta(a)\langle w \rangle)$ so $\theta(P_<) = P_\prec$ and $\theta$ is a morphism in $\LOs$.  Conversely, every morphism of the form $\fI_L(\theta)$ is injective since all homomorphisms in $\LOs$ are injective.   For injectivity of $\fI_L$ on objects, suppose $\fI_L(G,<,z) = (G,\ol G,\iota,\pi,s)$ and $\fI_L(H,\prec,w) = (H,\ol H, \epsilon,\rho,t)$ are the same object in $\BigL$.  Then $G = H$, and since $\iota = \rho$, $\iota(1) = z = w =\epsilon(1)$.  The positive cones $P_<$ and $P_\prec$ are equal since $t = s$, so $(G,<,z) = (H,\prec,w)$. To see $\fI_C$ is injective on objects, suppose $(G,c) \neq (G,d)$ in $\Circ$.  Since circular orderings are invariant under left multiplication, we may assume there are $a,b \in G$ such that $c(id,a,ab) \neq d(id,a,ab)$ so $f_c(a,b) \neq f_d(a,b)$.  Therefore $\fI_C(G,c) \neq \fI_C(G,d)$, proving (4).  Finally $(5)$ follows from $(1)$ and $(4)$.   
\end{proof}

Lemma \ref{useful-inclusion-properties} allows us to identify $\mathbf{LO_*}$ and $\mathbf{Circ}$ as subcategories of $\mathbf{BigLO_*}$ and $\mathbf{BigCirc}$ respectively.  Indeed, we can conclude that $\Circ$ and $\LOs$ are isomorphic (not just equivalent) to the subcategories of $\BigC$ and $\BigL$ consisting of objects in the image of $\fI_C$ and $\fI_L$, and morphisms consisting of all injective morphisms in the respective categories.  The next lemma shows that the equivalences $Q$ and $L$ from Proposition \ref{categorical-equivalence} are induced by the equivalences $\fQ$ and $\fL$ from Proposition \ref{big-categorical-equivalence}.

\begin{lemma}\label{induced-equivalence}
We have equality of functors $\fI_CQ = \fQ\fI_L:\LOs \to \BigC$ and a natural isomorphism of functors $\fI_LL \simeq \fL\fI_C:\Circ \to \BigL$.
\end{lemma}
\begin{proof}
Let $(G,<,z)$ be an object in $\LOs$.  Then $\fI_CQ(G,<,z) = (\ol G,\{f_{c_<}\})$ and $\fQ\fI_L(G,<,z) = (\ol G,\{f_s\})$ where $s:\ol G \to G$ is $s(g\langle z \rangle) = \ol g$.  Therefore $\fI_CQ(G,<,z) = \fQ\fI_L(G,<,z)$ by Lemma \ref{key-equivalence} since in the notation of Lemma \ref{key-equivalence}, $f_< = f_s$.  For any morphism $\theta:(G,<,z) \to (H,\prec,w)$ in $\LOs$, $\fI_CQ(\theta) = \fQ\fI_L(\theta) = \ol \theta:\ol G \to \ol H$.  Therefore $\fI_CQ =\fQ\fI_L$.

On the other hand, $\fI_LL(G,c) = \fI_L(\wt G_c,<_c,z_c)$, which we will denote by $\{(\wt G_c, \ol{\wt G}_c, \iota_c,\pi_c,s_c)\}$, and note that $\fL\fI_C(G,c) = \{(\wt G_{f_c}, G, \iota_{f_c}, \pi_{f_c},s_{f_c})\}$.  Define $\zeta_G:\wt G_c \to \wt G_{f_c}$ by $\zeta_G((n,a)) = (n,a)$.  Since the multiplication on both $\wt G_c$ and $\wt G_{f_c}$ are given by $(n,a)(m,b) = (n+m+f_c(a,b),ab)$, it is clear that $\zeta_G$ is a sectioned central extension isomorphism.  Therefore $\{\zeta_G\}:\fI_LL(G,c) \to \fL\fI_C(G,c)$ is an isomorphism in $\BigL$.  Now let $\phi:(G,c) \to (H,d)$ be a morphism in $\Circ$.  Then $\{\zeta_H\}\fI_LL(\phi)(n,a) = \fL\fI_C\{\zeta_G\}(\phi)(n,a) = (n,\phi(a))$ and $\fI_LL \simeq \fL\fI_C$.
\end{proof}

\begin{remark}
It is clear that the image of $\fI_L$ consists of singletons of sectioned central extensions (and similarly the image of $\fI_C$ consists of pairs $(G,S)$ where $S$ is a set containing a single 2-cocycle).  However, if we restricted our attention to the subcategories of $\mathbf{BigCirc}$ and $\mathbf{BigLO_*}$ consisting of sets of size 1, we would again be in a situation where amalgamation diagrams do not have colimits.
\end{remark}

\subsection{Amalgamated free products in $\BigC$ and $\BigL$}

We wish to show that colimits of amalgamation diagrams in the image of $\Circ$ and $\LOs$ inside $\BigC$ and $\BigL$ exist.  By Lemma \ref{useful-inclusion-properties}, it suffices to show that in $\BigC$, colimits of amalgamation diagrams exist when the objects in the diagrams are of the form $(G,S)$ where $S$ is a singleton, and all morphisms in the diagram are injective homomorphisms.

In $\BigC$, consider the amalgamation diagram $\vD = ((H,\{d\}), \{(\phi_i,(G_i,\{c_i \}))\}_{i \in I})$ where $\phi_i:H \to G_i$ is an injective homomorphism for all $i \in I$ that identifies $H$ with a subgroup $H_i \subset G$.  For the remainder of this section, let $\wt G_i$ denote $(\wt{G_i})_{c_i}$, $\wt{H_i}$ will denote the lift of $H_i$ with respect to the restriction of $c_i$, and $\wt{H}$ will denote $\wt{H}_d$.  Set
\[
\vG_{\vD} = *_{i \in I} G_i(H_i \overset{\phi_i}\cong H) \quad \text{and} \quad \vG_{\wt \vD} = *_{i \in I} \wt G_i(\wt H_i \overset{\wt \phi_i}\cong \wt H).
\]
Let $\delta_i:G_i \to \vG_{\vD}$ be the inclusion homomorphisms, and set $$T = \{f \in \Gamma^2(\vG_{\vD},\ZZ) \mid \delta_i^*f = c_i \text{ for all } i \in I \}.$$  

\begin{lemma} \label{BigC-colimit}
The object $(\vG_{\vD},T)$ with the morphisms $\delta_i:(G_i,\{c_i\}) \to (\vG_{\vD},T)$ is the colimit of $\vD$.
\end{lemma}
\begin{proof}
The maps $\delta_i:(G_i,\{c_i\}) \to (\vG_{\vD},T)$ are clearly morphisms in $\BigC$. Let $(B,S)$ be an object in $\BigC$ with morphisms $\psi_i:(G_i,\{c_i\}) \rightarrow (B, S)$ such that $\psi_i\phi_i = \psi_j \phi_j$ for all $i,j \in I$.  Let $\Psi:\vG_{\vD} \to B$ be the unique homomorphism arising from the universal property of the free product with amalgamation, so $\Psi\delta_i = \psi_i$ for all $i$.  It suffices to show $\Psi^*(S) \subset T$.  Let $s \in S$.  Then for each $i \in I$, $\delta_i^*\Psi^*(s) = \psi^*_i(s) = c_i$ so $\Psi^*(s) \in T$, completing the proof.
\end{proof}

The next lemma identifies the colimit, up to isomorphism, of the amalgamation diagram $\wt \vD$ in $\BigL$ obtained by applying the functor $\fL$ to $\vD$.

\begin{lemma} \label{BigL-colimit}
Let $f \in T$.  There is a group isomorphism $\Theta:\vG_{\wt \vD} \to(\wt \vG_{\vD})_f$ such that for each $i \in I$, 
\[
\Theta^{-1}\wt \delta_i: \wt{G_i} \to \vG_{\wt \vD}
\]
is the inclusion homomorphism.
\end{lemma}
\begin{proof}
Consider the maps $\wt \delta_i: \wt G_i \to (\wt{\vG}_{\vD})_f$ arising from lifts of the canonical inclusions.  Note that
\[\wt{\delta}_i \wt \phi_i = \wt{\delta_i \phi_i} = \wt{\delta_j \phi_j} = \wt{\delta}_j \wt{\phi}_j\]
for all $i, j \in I$, so the universal property of the free product with amalgamation $\vG_{\wt{\vD}}$ yields a map \linebreak $\Theta:\vG_{\wt \vD} \to (\wt \vG_{\vD})_f$.  Since $(\wt \vG_{\vD})_f$ is generated by $(1,0)$ along with $\{(0,g) \mid g \in G_i \text{ for some } i \in I\}$, $\Theta$ is surjective.  For injectivity, observe that $S \subset G_i$ is a set of right coset representatives of $G_i /\phi_i(H)$ if and only if $\{(0,s) \mid s \in S\}$ is a set of right coset representatives of $\wt G_i/\wt \phi_i(\wt H)$.  

Now suppose $w = (n,h)(0,g_1)\cdots(0,g_k)$ is the normal form of an element in $\vG_{\wt{\vD}}$ that satisfies $\Theta(w) = id$.  Then using the fact that $\Theta$ acts as $\wt{\delta}_i$ on each factor in the free product, we compute
\[
\Theta(w) = \left(n + \sum_{i=1}^k f(hg_1\cdots g_{i-1},g_i),hg_1\cdots g_k\right).
\]
The right hand side is the identity if and only if it is equal to $(0, id)$ (as an element of $\mathbb{Z} \times \vG_{\vD}$).  Since $hg_1\cdots g_k$ is the normal form of an element in $\vG_{\vD}$, we must have $h = g_1 = \cdots = g_k = id$.  Therefore $\sum_{i=1}^k f(hg_1\cdots g_{i-1},g_i) = 0$, so $n = 0$ and $\Theta$ is an isomorphism.  The fact that $\Theta^{-1}\wt\delta_i:\wt G_i \to \vG_{\wt \vD}$ is the inclusion homomorphism for each $i $ follows immediately from the construction of $\Theta$.
\end{proof}

We now have the machinery in place to prove our main result.  In what follows below, when $H$ is a subgroup of a circularly-ordered group $(G,c)$ we will write $(H,c)$ to denote the subgroup $H$ equipped with the restriction ordering arising from $c$.

\begin{textmain}
\label{main-coro}
Suppose $(G_i,c_i)$ are circularly-ordered groups for $i \in I$, each equipped with a subgroup $H_i \subset G_i$ and an order-preserving isomorphism $\phi_i:(H,d) \to (H_i,c_i)$ from a circularly-ordered group $(H,d)$.  The following are equivalent:
\begin{enumerate}
\item The group $*_{i \in I} G_i  (H_i \stackrel{\phi_i}{\cong} H)$ admits a circular ordering $c$ which extends the orderings $c_i$ of $G_i$ for $i \in I$.
\item The group $*_{i \in I} \widetilde G_i (\widetilde H_i \stackrel{\widetilde \phi_i}{\cong}\widetilde{H})$ admits a left ordering $<$ which extends each of the left orderings $<_{c_i}$ of $\widetilde G_i$ for $i \in I$.
\end{enumerate}
\end{textmain}
\begin{proof}
Let $\vD = ((H,\{f_d\}), \{(\phi_i,(G_i,\{f_{c_i}\}))\}_{i \in I})$ be the amalgamation diagram in $\BigC$, and let $\wt \vD$ be the amalgamation diagram in $\BigL$ obtained by applying the functor $\fL$ to $\vD$.  Since $\vL$ is an equivalence of categories by Lemma \ref{big-categorical-equivalence} and $\vD$ has a colimit by Lemma \ref{BigC-colimit}, the diagram $\wt \vD$ has a colimit.  By Lemma \ref{BigC-colimit}, the colimit of $\vD$ is the object $(\vG_\vD,\{f_\alpha\}_{\alpha \in \vA})$ along with the inclusion homomorphisms $\delta_i:G_i \to \vG_\vD$.  By Lemma \ref{BigL-colimit}, the colimit of $\wt \vD$ is given by the object $\{(\vG_{\wt\vD}, \vG_\vD, \iota,\pi,s_\alpha)\}_{\alpha \in \vA}$ where $\iota(1) = (1,id)$,
\[
\pi((n,h)(0,g_1)\cdots (0,g_k)) = hg_1\cdots g_k,
\]
and $s_\alpha(hg_1\cdots g_k) = (-\sum_{i=1}^k f_\alpha(hg_1\cdots g_{i-1},g_i),h)(0,g_1)\cdots(0, g_k)$.  The morphisms
\[
\{\varphi_{i,\alpha}\}_{\alpha \in \vA}:\fL(G_i,\{f_{c_i}\}) \to \{(\vG_{\wt\vD}, \vG_\vD, \iota,\pi,s_\alpha)\}_{\alpha \in \vA}
\]
are given by $\varphi_{i,\alpha} = \sigma_i$ for all $\alpha \in \vA$, where $\sigma_i:(\wt G_i)_{f_{c_i}} \to \vG_{\wt\vD}$ is the inclusion homomorphism.

Suppose (1) holds.  Then there is some $\alpha \in \vA$ such that $f_\alpha = f_c$.  By Lemma \ref{induced-equivalence}, 
\[
\fL(\vG_\vD,\{f_c\}) = \{(\vG_{\wt\vD},\vG_\vD,\iota,\pi,s_c)\}\]
gives rise to a left ordering $<$ on $\vG_{\wt\vD}$.  Since the morphisms $\{\sigma_i \}:\fL(G_i,\{f_{c_i}\}) \to \{(\vG_{\wt\vD},\vG_\vD,\iota,\pi,s_c)\}$ are such that $\sigma_i$ is injective, $\sigma_i$ is an order preserving homomorphism by Lemma \ref{useful-inclusion-properties} (3).  Since the $\sigma_i$ are the inclusion homomorphisms, $<$ extends each of the $<_{c_i}$.  Similarly, if $(2)$ holds, then there is some $\beta\in \vA$ such that $\{(\vG_{\wt\vD},\vG_\vD,\iota,\pi,s_\beta)\}$ corresponds to the left order $<$ on $\vG_{\wt \vD}$.  Then $\fQ\{(\vG_{\wt\vD},\vG_\vD,\iota,\pi,s_\beta)\} = (\vG_\vD,\{f_\beta\})$ and $c = c_{<_{f_\beta}}$ is the desired circular ordering on $\vG_\vD$ extending each of the $c_i$ on $G_i$.
\end{proof}

\begin{remark}
\label{core-remarks}
We close out the section with the following remarks about Theorem \ref{main-coro}.
\begin{enumerate}
\item Defining $\BigC$ (and $\BigL$) as we have done gives rise to colimits of amalgamation diagrams in $\Circ$ (and $\LOs$) that contain (as part of their defining data) {\it all} circular orderings (and left orderings) on free products with amalgamation that extend the orderings on the factor groups.
\item The proof provides a bijection between circular orderings $c$ on $\vG_\vD$ extending each of the $c_i$ on $G_i$ and left orderings $<$ on $\vG_{\wt\vD}$ extending each of the $<_{c_i}$ on $\wt G_i$. 
\item The fact that all the sectioned central extensions in the colimit of $\wt D$ have the same underlying central extension implies that all circular orderings $c$ extending each of the $c_i$ give the same cohomology class $[f_c] \in H^2(\vG_\vD;\ZZ)$.
\end{enumerate}
\end{remark}

\section{Lifting, quotients and compatible normal families of orderings}
\label{compatible-section}

The goal of this section is to explore necessary and sufficient conditions that $*_{i \in I} G_i  (H_i \stackrel{\phi_i}{\cong} H)$ admits a circular ordering $c$ which extends the orderings $c_i$ of $G_i$ for each $i$, which can be stated in terms of circular orderings (or families of circular orderings) of the groups $G_i$.   First we recall such conditions in the case of left-orderable groups, and generalize the terminology used there to the case of circular orderings.

Denote the collection of all left orderings of a group $G$ by $\mathrm{LO}(G)$, and the collection of all circular orderings of $G$ by $\mathrm{CO}(G)$.  Appropriately topologized, each becomes a compact Hausdorff space \cite{BS15}\cite{Sikora04}.  These spaces also each come equipped with a $G$-action by homeomorphisms, defined as follows.  Recall that if $(G, <)$ is a left-ordered group, for each $h \in G$ there is a left ordering $<^h$ defined by $g_1 <^h g_2$ if and only if $g_1h<g_2h$.  If the positive cone of the ordering $<$ is $P$, then the positive cone of $<^h$ is $h^{-1}Ph$.  Similarly if $(G,c)$ is a circularly-ordered group, then for each $h \in G$ there is a circular ordering $c^h$ defined by $c^h(g_1, g_2, g_3) = c(g_1h, g_2h, g_3h)$ \cite{BS15}. Note that since circular orderings and left orderings are left-invariant, $<^h$ and $c^h$ are simply the orders obtained by pulling back $<$ and $c$ by the automorphism of $G$ given by conjugation by $h$. A \textit{normal family} of left orderings (resp. circular orderings) of a group $G$ is a set $\mathcal{N} \subset \mathrm{LO}(G)$ (resp. $\mathcal{N} \subset \mathrm{CO}(G)$) that is invariant under the $G$-action.  

Let $\{(G_i,<_i)\}_{i \in I}$ be a collection of left-ordered groups with positive cone $P_i \subset G_i$, and let $H$ be a group.  For each $i$, let $\phi_i:H \to G_i$ be an injective homomorphism with image $H_i \subset G_i$. We say the collection $\{<_i\}_{i \in I}$ is {\it compatible with $\{\phi_i\}_{i \in I}$} if for all $i,j \in I$, $\phi_i^*P_i = \phi_j^*P_j$, or equivalently, if $<_i^{\phi_i}$ and $<_j^{\phi_j}$ are the same left order on $H$. Given sets $S_i \subset \mathrm{LO}(G_i)$, we say $\{S_i\}_{i \in I}$ is {\it compatible with $\{\phi_i\}_{i \in I}$} if for all $i,j \in I$, $\phi_i^*S_i = \phi_j^*S_j$.


We have the following theorem from Bludov and Glass, which is Theorem A in \cite{BV09}.

\begin{theorem}
\label{bludov-glass}
Suppose that $\{(G_i, <_i)\}_{i \in I}$ are left-ordered groups, $H$ is a group and for each $i$, $\phi_i:H \to G_i$ is an injective homomorphism with image $H_i \subset G_i$.  The group $*_{i \in I} G_i  (H_i \stackrel{\phi_i}{\cong} H)$ is left-orderable (via an ordering extending each of the $<_i$) if and only if the collection $\{ <_i \}_{i \in I}$ is compatible with $\{ \phi_i \}_{i \in I}$ and there exist normal families $\mathcal{N}_i \subset \mathrm{LO}(G_i)$ with $<_i \in \mathcal{N}_i$ for all $i$ such that $\{ \mathcal{N}_i \}_{i \in I}$ is compatible with $\{ \phi_i \}_{i \in I}$.
\end{theorem}

We define compatibility of circular orderings similarly. Suppose we have a collection of circularly-ordered groups $\{(G_i,c_i)\}_{i \in I}$, a group $H$, and for each $i$ an injective homomorphism $\phi_i:H \to G_i$ with image $H_i \subset G_i$. We say the collection $\{c_i\}_{i \in I}$ is {\it compatible with $\{\phi_i\}_{i \in I}$} if for all $i,j \in I$, $\phi_i^*c_i = \phi_j^*c_j$. Given sets $S_i \subset \mathrm{CO}(G_i)$, we say $\{S_i\}_{i \in I}$ is {\it compatible with $\{\phi_i\}_{i \in I}$} if for all $i,j \in I$, $\phi_i^*S_i = \phi_j^* S_j$.



Our goal is to prove an analogous theorem to Theorem \ref{bludov-glass} for circularly-orderable groups.  Recall the lifting and quotienting constructions from Section \ref{background}.

\begin{lemma}\label{normal-descends}
Let $(G,<,z)$ be an object of $\LOs$, and let $g \in G$. Then $c_{<^g} = c_<^{g \langle z \rangle}$ as circular orderings on $\ol G$.  
\end{lemma}
\begin{proof}
Given $(g_1\langle z \rangle, g_2 \langle z \rangle, g_3 \langle z \rangle) \in \ol G^3 \sm \Delta(\ol G)$, note that 
$c_{<^g}(g_1\langle z \rangle, g_2 \langle z \rangle, g_3 \langle z \rangle) = sign(\sigma)$, where $\sigma$ is the unique permutation such that $\widehat{g_{\sigma(1)}} <^g \widehat{g_{\sigma(2)}} <^g \widehat{g_{\sigma(3)}}$.   Here, $\widehat{g_{\sigma(i)}}$ is the unique element of $g_{\sigma(i)}\langle z \rangle$ such that $id \leq^g \widehat{g_{\sigma(i)}} <^g z$, i.e. it is the minimal representative with respect to $<^g$.   It follows that $id \leq g^{-1} \widehat{g_{\sigma(i)}} g < z$ and thus $g^{-1} \widehat{g_{\sigma(i)}} g = \ol{g^{-1}g_{\sigma(i)}g}$, the minimal representative with respect to $<$.  Thus $\widehat{g_{\sigma(1)}} <^g\widehat{g_{\sigma(2)}} <^g \widehat{g_{\sigma(3)}}$, which is equivalent to $g^{-1} \widehat{g_{\sigma(1)}}g < g^{-1} \widehat{g_{\sigma(2)}}g < g^{-1}\widehat{g_{\sigma(3)}}g$, is equivalent to
\[
\ol{g^{-1}g_{\sigma(1)}g} < \ol{g^{-1}g_{\sigma(2)}g} < \ol{g^{-1} g_{\sigma(3)}g}.
\]

On the other hand,
\[
c^{g \langle z \rangle}_{<} (g_1\langle z \rangle, g_2 \langle z \rangle, g_3 \langle z \rangle)= c_{<} (g^{-1}g_1g\langle z \rangle, g^{-1}g_2 g\langle z \rangle, g^{-1}g_3g \langle z \rangle) =sign(\tau),
\]
where $\tau$ is the unique permutation such that $\ol{g^{-1}g_{\tau(1)}g} < \ol{g^{-1}g_{\tau(2)} g} < \ol{g^{-1}g_{\tau(3)}g}$.  Thus $\sigma = \tau$ and the lemma follows.
\end{proof}

\begin{lemma}\label{ccc}
Let $(G,c)$ be a circularly-ordered group, and let $g \in G$. Then $[f_{c^g}] = [f_c]$ in $H^2(G;\ZZ)$.
\end{lemma}
\begin{proof}
Let $(\wt G_c, <_c,z_c)$ be the lift of $(G,c)$, and let $\tilde g = (0,g) \in \wt G$. Then by Lemma \ref{normal-descends}, $f_{c_{<_c}^{\tilde g\langle z_c \rangle}}= f_{c_{<_c^{\tilde g}}}$. Let $\nu:G \to \wt G/\langle z_c \rangle$ be the isomorphism given by $\nu(h) = (0,h)\langle z_c \rangle$. Then $\nu^* f_{c_{<_c}} = f_c$ by Lemma \ref{key-equivalence}.

The two cocycles $f_{c_{<_c}}$ and $f_{c_{<_c^{\tilde g}}}$ both correspond to the central extension
\[
1 \lra \ZZ \overset{\iota}\lra \wt G \lra \wt G/\langle z_c \rangle \lra 1
\]
so $\left [f_{c_{<_c}}\right ] = \left [f_{c_{<_c^{\tilde g}}}\right ]$ in $H^2(\wt G/\langle z_c \rangle; \ZZ)$. Indeed, let $s_1,s_2:\wt G/\langle z_c \rangle \to \wt G$ be the sections given by
\[
s_1((0,h)\langle z_c \rangle) = \ol{(0,h)} \quad \text{and} \quad s_2((0,h)\langle z_c \rangle) = \widehat{(0,h)}
\]
where $\ol{(0,h)}$ and $\widehat{(0,h)}$ are the minimal representatives of $(0,h)$ with respect to $<_c$ and $<_c^{(0,g)}$ respectively. Then $f_{c_{<_c}}$ is the cocycle obtained from the section $s_1$, and $f_{c_{<_c^{\tilde g}}}$ is the cocycle obtained from $s_2$ (see Remark \ref{derived-cocycle}).

We now claim that as cocycles on $G$, $\nu^*f_{c_{<_c}^{\tilde g\langle z_c \rangle}} = f_{c^g}$. Let $a,b \in G$. Then
\begin{align*}
\nu^*f_{c_{<_c}^{\tilde g\langle z_c \rangle}}(a,b) &= f_{c_{<_c}^{\tilde g\langle z_c \rangle}}((0,a)\langle z_c \rangle,(0,b)\langle z_c\rangle) \\
&= f_{c_{<_c}}((0,g^{-1}ag)\langle z_c \rangle,(0,g^{-1}bg)\langle z_c\rangle) \\
&= \nu^*f_{c_{<_c}}(g^{-1}ag,g^{-1}bg) \\
&= f_c(g^{-1}ag,g^{-1}bg) \\
&= f_{c^g}(a,b).
\end{align*}
Putting all of this together, in $H^2(G;\ZZ)$ we have
\[
[f_c] = \nu^*\left [f_{c_{<_c}}\right] = \nu^*\left[f_{c_{<_c^{\tilde g}}}\right] = \nu^*\left[f_{c_{<_c}^{\tilde g\langle z_c \rangle}}\right] = [f_{c^g}]
\]
completing the proof.
\end{proof}

Given a circularly-orderable group $G$, we call a collection of circular orderings $S \subset \mathrm{CO}(G)$ {\it cohomologically constant} if the Euler class function $e:S \to H^2(G;\ZZ)$ given by $e(c) = [f_c]$ is constant.  That is, $[f_c] = [f_{c'}]$ for all $c,c' \in S$. 

We are now ready to prove one direction of the analogous result to Theorem \ref{bludov-glass} for circularly-orderable groups.

\begin{proposition}\label{bg-circ-forward}
Suppose that $\{(G_i,c_i)\}_{i \in I}$ are circularly-ordered groups, $H$ is a group and for each $i$, $\phi_i:H \to G_i$ is an injective homomorphism with image $H_i \subset G_i$. If there is a circular ordering $c$ on $*_{i \in I}G_i(H_i \overset{\phi_i}\cong H)$ extending each of the $c_i$, then there exist cohomologically constant normal families \linebreak $\vR_i \subset \mathrm{CO}(G_i)$ such that $c_i \in \vR_i$ for all $i \in I$ and the collections $\{c_i\}_{i \in I}$ and $\{\vR_i\}_{i \in I}$ are compatible with $\{\phi_i\}_{i \in I}$. 
\end{proposition}
\begin{proof}
To ease notation, let $\vG = *_{i \in I}G_i(H_i \overset{\phi_i}\cong H)$ and for each $i$, let $\delta_i:G_i \to \vG$ be the inclusion homomorphism. Since $c$ extends each of the $c_i$, we have $\delta_i^*c = c_i$ for all $i \in I$. Then for $i,j \in I$,
\[
\phi_i^*c_i = \phi_i^*\delta_i^*c = \phi_j^*\delta_j^*c = \phi_j^*c_j
\]
so $\{c_i\}_{i \in I}$ is compatible with $\{\phi_i\}_{i \in I}$.

To construct the families $\vR_i \subset \mathrm{CO}(G_i)$, first construct a family $\vR \subset \mathrm{CO}(\vG)$ by $\vR = \{c^w: w \in \vG\}$. For each $i \in I$, define $\vR_i = \delta_i^*\vR$. Note that $\vR$ is cohomologically constant by Lemma \ref{ccc}, so $\vR_i$ is cohomologically constant. Since $c_i = \delta_i^*c$, we have $c_i \in \vR_i$. For compatibility, let $i,j \in I$. Since $\phi_i^*\delta_i^*c^w = \phi_j^*\delta_j^*c^w$, $\phi_i^*\vR_i = \phi_j^*\vR_j$ and $\{\vR_i\}_{i \in I}$ is compatible with $\{\phi_i\}_{i \in I}$. Finally, for normality let $\delta_i^*c^w \in \vR_i$, and let $g \in G_i$. Then $(\delta_i^*c^w)^g = \delta_i^*c^{\delta_i(g)w}$, completing the proof.
\end{proof}

We now shift our focus to proving a partial converse to Proposition \ref{bg-circ-forward}, which is Proposition \ref{bg-circ-back} below.

Suppose that $\{G_i,c_i\}_{i \in I}$ are circularly-ordered groups, $H$ is a group and for each $i$, $\phi_i:H \to G_i$ is an injective homomorphism with image $H_i \subset G_i$. Assume that for each $i$ there is a cohomologically constant normal family $\vR_i \subset \mathrm{CO}(G_i)$ such that $c_i \in \vR_i$, and both $\{c_i\}_{i \in I}$ and $\{\vR_i\}_{i \in I}$ are compatible with $\{\phi_i\}_{i \in I}$.

The strategy to prove Proposition \ref{bg-circ-back} is to construct normal families $\wt \vR_i \subset \mathrm{LO}(\wt G_i)$, show that $\{\wt\vR_i\}_{i \in I}$ is compatible with $\{\wt\phi_i\}_{i \in I}$, and use Theorem \ref{bludov-glass} and Theorem \ref{main-coro} to conclude that $*_{i \in I}G_i(H_i \overset{\phi_i}\cong H)$ is circularly-orderable via a circular ordering extending each of the $c_i$. We will first construct the $\wt \vR_i$ and show it is normal.

\begin{construction}\label{normal-family-lift}
Suppose $(G,c)$ is a circularly-ordered group, and let $\vR \subset \mathrm{CO}(G)$ be a cohomologically constant normal family of circular orderings such that $c \in \vR$. Let $\vR = \{c_\alpha : \alpha \in \vA\}$. For each $\alpha \in \vA$, choose a function $d_\alpha:G \to \ZZ$ such that 
\[
f_c(a,b) - f_{c_\alpha}(a,b) = d_\alpha(a) + d_\alpha(b) - d_{\alpha}(ab)
\]
for all $a,b \in G$. Such a function exists since $\vR$ is cohomologically constant.  Let $(\wt G,<,z)$ be the lift of $(G,c)$, and recall the positive cone $P \subset \wt G$ of $<$ is given by $P \cup \{id\} = \{(n,a) \in \wt G: n \geq 0\}$. For each $\alpha \in \vA$ and $\varphi \in H^1(G;\ZZ) = \operatorname{Hom}(G,\ZZ)$, define the positive cone $P_{\alpha,\varphi} \subset \wt G$ by
\[
P_{\alpha,\varphi} \cup\{id\} = \{(n,a) \in \wt G: n + d_\alpha(a) + \varphi(a) \geq 0\}.
\]
The set $P_{\alpha,\varphi}$ can be seen to be a positive cone as follows. Let $(\wt G^\alpha,<_{c_\alpha},z_{c_\alpha})$ be the lift of $(G,c_\alpha)$ with positive cone $Q_\alpha$. There is an equivalence of short exact sequences $\Phi_{\alpha,\varphi}:\wt G \to \wt G^\alpha$ given by
\[
\Phi_{\alpha,\varphi}(n,a) = (n + d_\alpha(a) + \varphi(a),a).\]
Then $\Phi_{\alpha,\varphi}^*Q_\alpha = P_{\alpha,\varphi}$.

Define the family $\wt \vR \subset \mathrm{LO}(\wt G)$ to be the set $
\wt \vR = \{P_{\alpha,\varphi}:\alpha \in \vA, \varphi \in H^1(G;\ZZ)\}.
$
\end{construction}
\begin{remark}
\label{construction-remark}
Note that if $d'_\alpha:G \to \ZZ$ is another choice of function so that
\[
f_c(a,b) - f_{c_\alpha}(a,b) = d'_\alpha(a) + d'_\alpha(b) - d'_\alpha(ab),\]
then $d_\alpha - d'_\alpha \in H^1(G;\ZZ)$. Therefore $\wt\vR$ does not depend on our choices of $d_\alpha$.

Furthermore, every equivalence of short exact sequences $\Phi:\wt G^\alpha \to \wt G$ is of the form $\Phi_{\alpha,\varphi}$, so we are constructing the family $\wt \vR$ by pulling back the positive cones $Q_\alpha$ via every equivalence of short exact sequences $\wt G \to \wt G^\alpha$.  In fact we could have equivalently defined $\wt \vR$ as the set of left orders $\prec$ on $\wt G$ such that \linebreak $\eta_GQ(\wt G,\prec,z)= (G,c_{\alpha})$ for some $\alpha \in \mathcal{A}$, where $Q$ is the quotient functor from Lemma \ref{functors} and $\eta_G:\wt G/\langle z \rangle \to G$ is the isomorphism from Lemma \ref{key-equivalence}.
\end{remark}

The next two lemmas prove that $\wt \vR$ is a normal family.
\begin{lemma}\label{unlikely-homomorphism}
Let $(G,c)$ be a circularly-ordered group, and let $\vR = \{c_\alpha:\alpha \in \vA\} \subset \mathrm{CO}(G)$ be a cohomologically constant normal family such that $c \in \vR$. Let $g \in G\sm\{id\}$ and suppose $\alpha,\beta \in \vA$ are such that $f_{c_\alpha}(a,b) = f_{c_\beta}(gag^{-1},gbg^{-1})$ for all $a,b \in G$. Choose $d_\alpha,d_\beta:G \to \ZZ$ as in Construction \ref{normal-family-lift}. Then $\varphi_{g,\alpha,\beta}:G \to \ZZ$ given by
\[
\varphi_{g,\alpha,\beta}(a) = -1 + f_c(g,a) + f_c(ga,g^{-1}) + d_\beta(gag^{-1}) - d_\alpha(a)
\]
is a homomorphism.
\end{lemma}
\begin{proof}
Let $a,b \in G$.  From the definition of $d_\alpha$ and $d_\beta$ we have
\begin{align*}
d_\alpha(a) + d_\alpha(b) - d_\alpha(ab) &= f_c(a,b) - f_{c_\alpha}(a,b),  \quad \text{and} \quad \\
d_\beta(gag^{-1}) + d_\beta(gbg^{-1}) - d_\beta(gabg^{-1}) &= f_c(gag^{-1},gbg^{-1}) - f_{c_\beta}(gag^{-1},gbg^{-1}).
\end{align*}
Note also that since $g \neq id$, $f_c(g^{-1},g) = 1$. In the computation below we will use these facts freely, rearrange terms as necessary, and enclose in square braces any terms that are to be replaced using the cocycle condition. We calculate:
\begin{align*}
&\varphi_{g,\alpha,\beta}(ab) - \varphi_{g,\alpha,\beta}(a) - \varphi_{g,\alpha,\beta}(b)
\\
& \quad= -1 + f_c(g,ab) + f_c(gab,g^{-1}) + d_\beta(gabg^{-1}) - d_\alpha(ab) \\
& \quad \quad +1 - f_c(g,a) - f_c(ga,g^{-1}) - d_\beta(gag^{-1}) + d_\alpha(a) \\
& \quad \quad +1 - f_c(g,b) - f_c(gb,g^{-1}) - d_\beta(gbg^{-1}) + d_\alpha(b) \\
&\quad = 1 + f_c(g,ab) + [f_c(gab,g^{-1})] - f_c(g,a) - f_c(ga,g^{-1}) - f_c(g,b) - f_c(gb,g^{-1})\\
&\quad \quad + f_c(a,b) - f_{c_\alpha}(a,b) - f_c(gag^{-1},gbg^{-1}) + f_{c_\beta}(gag^{-1},gbg^{-1})\\
&\quad = 1 + [f_c(a,b) - f_c(ga,b) + f_c(g,ab) - f_c(g,a)] + [f_c(b,g^{-1}) - f_c(gb,g^{-1}) - f_c(g,b)]\\
&\quad \quad + f_c(ga,bg^{-1}) - f_c(ga,g^{-1}) - f_c(gag^{-1},gbg^{-1})\\
&\quad = 1 + [-f_c(g,bg^{-1}) +f_c(ga,bg^{-1}) - f_c(gag^{-1},gbg^{-1})] - f_c(ga,g^{-1}) \\
&\quad = 1 + -f_c(gag^{-1},g) - f_c(ga,g^{-1}) \\
&\quad = [f_c(g^{-1},g) - f_c(gag^{-1},g) + f_c(ga,id) - f_c(ga,g^{-1})] \\
&\quad = 0,
\end{align*}
completing the proof.
\end{proof}

\begin{lemma} \label{normal-lift-is-normal}
Let $(G,c)$ be a circularly-ordered group, and let $\vR = \{c_\alpha:\alpha \in \vA\} \subset \mathrm{CO}(G)$ be a cohomologically constant normal family such that $c \in \vR$. Let $(\wt G,<,z)$ be the lift of $(G,c)$. The set $\wt\vR \subset \mathrm{LO}(G)$ from Construction \ref{normal-family-lift} is a normal family of left orderings on $\wt G$.
\end{lemma}
\begin{proof}
Consider a positive cone $P_{\alpha,\varphi} \in \wt \vR$ and let $g \in G\sm\{id\}$.  Since $(m,id) \in \wt G$ is central and $(m,g) = (m,id)(0,g)$, it suffices to show $(0,g)P_{\alpha,\varphi}(0,g)^{-1} \in \wt\vR$.

Let $\beta \in \vA$ be such that $f_{c_\alpha}(a,b) = f_{c_\beta}(gag^{-1},gbg^{-1})$.  Define the homomorphism $\psi:G \to \ZZ$ by $\psi(a) = \varphi(a) - \varphi_{g,\alpha,\beta}(a)$ (where $\varphi_{g,\alpha,\beta}$ is defined in Lemma \ref{unlikely-homomorphism}).  We will show $(0,g)P_{\alpha,\varphi}(0,g)^{-1} = P_{\beta,\psi}$.

Let $(n,a) \in P_{\alpha,\varphi} \cup \{id\}$, so $n + d_\alpha(a) + \varphi(a) \geq 0$.  Then 
\[
(0,g)(n,a)(0,g)^{-1} = (n-1 + f_c(g,a) + f_c(ga,g^{-1}),gag^{-1}).
\]
To check that $(0,g)(n,a)(0,g)^{-1} \in P_{\beta,\psi} \cup \{id\}$ we have
\begin{align*}
&n-1 + f_c(g,a) + f_c(ga,g^{-1}) + d_\beta(gag^{-1}) + \psi(a) \\
&\quad \quad = n - 1 + f_c(g,a) + f_c(ga,g^{-1}) + d_\beta(gag^{-1}) + \varphi(a) \\
&\quad \quad \quad \quad - (-1 + f_c(g,a) + f_c(ga,g^{-1}) + d_\beta(gag^{-1}) - d_\alpha(a)) \\
& \quad \quad = n + d_\alpha(a) + \varphi(a) \geq 0.
\end{align*}
Therefore $(0,g)(n,a)(0,g)^{-1} \in P_{\beta,\psi} \cup\{id\}$ and $(0,g)P_{\alpha,\varphi}(0,g)^{-1} \subset P_{\beta,\psi}$.  Since both sets are positive cones, we have $(0,g)P_{\alpha,\varphi}(0,g)^{-1} = P_{\beta,\psi}$, completing the proof.
\end{proof}

The next proposition provides a partial converse to Proposition \ref{bg-circ-forward}.

\begin{proposition}\label{bg-circ-back}
Suppose that $\{(G_i,c_i)\}_{i \in I}$ are circularly-ordered groups, $H$ is a group and for each $i$, $\phi_i:H \to G_i$ is an injective homomorphism with image $H_i \subset G_i$. Suppose further that for each $i \in I$, \[\phi_i^*:H^1(G_i;\ZZ) \to H^1(H;\ZZ)\] is surjective. If there exist cohomologically constant normal families $\vR_i \subset \mathrm{CO}(G_i)$ such that $c_i \in \vR_i$ for all $i \in I$ and the collections $\{c_i\}_{i \in I}$ and $\{\vR_i\}_{i \in I}$ are compatible with $\{\phi_i\}_{i \in I}$, then there is a circular ordering $c$ on $*_{i \in I}G_i(H_i \overset{\phi_i}\cong H)$ extending each of the $c_i$.
\end{proposition}
\begin{proof}
Fix $i \in I$, and suppose $c_i = c_\gamma \in \vR_i$. Then $d_\gamma:G \to \ZZ$ is necessarily a homomorphism. Let $(\wt G_i,<_i,z_i)$ be the lift of $(G_i,c_i)$ with positive cone $P_i$. Then $P_i = P_{\gamma,-d_\gamma} \in \wt \vR_i$. By Theorem \ref{main-coro}, Theorem \ref{bludov-glass}, and Lemma \ref{normal-lift-is-normal}, it suffices to show that both $\{<_i\}_{i \in I}$ and $\{\wt \vR_i\}_{i \in I}$ (where $\wt \vR_i$ is defined in Construction \ref{normal-family-lift}) are compatible with $\{\wt\phi_i\}_{i \in I}$. Let $(\wt H,\prec,z)$ be the lift of $(H,\phi_i^*c_i)$, which is independent of $i$ by compatibility, and recall that $\wt\phi_i:\wt H \to \wt G_i$ is given by $\wt\phi_i(n,h) = (n,\phi_i(h))$.

Fix $i,j \in I$. Then $\wt\phi_i^*P_i \cup\{id\} = \wt\phi_j^*P_j\cup\{id\} = \{(n,h) \in \wt H: n \geq 0\}$.  Therefore $\{<_i\}_{i \in I}$ is compatible with $\{\wt\phi_i\}_{i \in I}$.

For compatibility of $\{\wt \vR_i\}_{i \in I}$, let $c_\alpha \in \vR_i$, and choose $c_\beta \in \vR_j$ such that $\phi_i^*c_\alpha = \phi_j^*c_\beta$.  Let $\varphi \in H^1(G_i,\ZZ)$.  We first claim that $\psi:H \to \ZZ$ given by $\psi(h) = d_\alpha\phi_i(h) + \varphi\phi_i(h) - d_\beta\phi_j(h)$ is a homomorphism.  Since $\varphi\phi_i$ is a homomorphism we have
\begin{align*}
\psi(ab) - \psi(a)-\psi(b) &= d_\alpha\phi_i(ab) - d_\beta\phi_j(ab) - d_\alpha\phi_i(a) + d_\beta\phi_j(a) - d_\alpha\phi_i(b) + d_\beta\phi_j(b) \\
&= \phi_i^*f_{c_\alpha}(a,b) - \phi_i^*f_{c_i}(a,b) -\phi_j^*f_{c_\beta}(a,b) + \phi_j^*f_{c_j}(a,b) \\
&= 0
\end{align*}
since $\phi_i^*f_{c_\alpha} = \phi_j^*f_{c_\beta}$ and $\phi_i^*f_{c_i} = \phi_j^*f_{c_j}$.  Choose $\hat\psi \in H^1(G_j,\ZZ)$ with the property that $\phi_j^*\hat\psi = \psi$. It suffices to show $\wt\phi_i^*P_{\alpha,\varphi} = \wt\phi_j^*P_{\beta,\hat\psi}$.  

Let $(n,h) \in \wt\phi_i^*P_{\alpha,\varphi} \cup\{id\}$, so $n + d_\alpha\phi_i(h) + \varphi\phi_i(h) \geq 0$.  We want to show $(n,h) \in \wt\phi_j^*P_{\beta,\hat\psi} \cup \{id\}$.  We have
\begin{align*}
n + d_\beta\phi_j(h) + \hat\psi\phi_j(h) &= n + d_\beta\phi_j(h) + \psi(h) \\
&= n + d_\beta\phi_j(h) + d_\alpha\phi_i(h) + \varphi\phi_i(h) - d_\beta\phi_j(h) \\
&= n + d_\alpha\phi_i(h) + \varphi\phi_i(h)\\
&\geq 0.
\end{align*}
Therefore $\wt\phi_i^*P_{\alpha,\varphi} \subset \wt\phi_j^*P_{\beta,\hat\psi}$ and since both are positive cones in $H$, they are equal, completing the proof.
\end{proof}

However the techniques of Propostion \ref{bg-circ-back} do not tell the whole story.  In the example that follows, we produce groups $G_1$ and $G_2$, injective homomorphisms $\phi_i :H \rightarrow G_i$, and equip the $G_i$'s with a circular orderings $c_i$ compatible with $\{\phi_1, \phi_2\}$.  Further, we equip the groups $G_1$ and $G_2$ with normal, compatible, cohomologically constant families of circular orderings $\mathcal{R}_1$ and $\mathcal{R}_2$ containing $c_1$ and $c_2$ respectively.  However the hypothesis that $\phi_i^*:H^1(G_i;\ZZ) \to H^1(H;\ZZ)$ be surjective for each $i$ will fail in our setup, and consequently the families $\wt{\mathcal{R}}_1$ and $\wt{\mathcal{R}}_2$ produced by Construction \ref{normal-family-lift} are not compatible with $\{ \wt{\phi}_1, \wt{\phi}_2\}$.  Nevertheless, $*_{i \in I}G_i(H_i \overset{\phi_i}\cong H)$ admits a circular ordering $c$ extending each of the $c_i$.

\begin{example} Set $G_1 = \mathbb{Z} \times \mathbb{Z}_n$ for some $n \geq 2$ and $G_2 = \mathbb{Z} \rtimes \mathbb{Z}_{2}$, where the action of $\mathbb{Z}_{2}$ on $\mathbb{Z}$ is multiplication by $-1$.  Let $H$ be an infinite cyclic group generated by $t$, and define $\phi_i : H \rightarrow G_i$ by $\phi_i(t) = (1,0)$, where $(1, 0)$ is understood as either an element of $G_1$ or $G_2$ depending on the subscript of $\phi_i$.

Define $\psi: \ZZ_n \rightarrow S^1$ by $\psi(1) = e^{2 \pi i/n}$ and let $c$ denote the standard ordering of $S^1$.  Equip $\ZZ_n$ with the circular ordering $\psi^*c$ and let $\mathcal{R}_1$ consist of the two circular orderings of $G_1$ that arise lexicographically from the short exact sequence
\[ 1 \rightarrow H \stackrel{\phi_1}{\rightarrow} G_1 \rightarrow \ZZ_n \rightarrow 1
\]
using the two standard linear orderings of $H$ and the circular ordering $\psi^*c$ of $\ZZ_n$.  Fix $c_1$ to be the circular ordering of $\mathcal{R}_1$ arising from the linear ordering of $H$ satisfying $t> id$.  Let $\mathcal{R}_2$ be the two circular orderings of $G_2$ that arise lexicographically from the short exact sequence 
\[ 1 \rightarrow H \stackrel{\phi_2}{\rightarrow} G_2 \rightarrow \ZZ_2 \rightarrow 1
\]
and fix $c_2$ to be the circular ordering arising from a choice of linear ordering of $H$ satisfying $t>id$. 

By construction, the families $\mathcal{R}_1$ and $\mathcal{R}_2$ are normal and compatible with the maps $\{ \phi_1, \phi_2\}$.  Moreover they are cohomologically constant: in the case of $\mathcal{R}_2$ this follows from an application of Lemma \ref{ccc}, in the case of $\mathcal{R}_1$ this follows from a direct analysis of the lifts and their corresponding short exact sequences.  Now let us analyze the families $\wt{\mathcal{R}}_i$ that arise from an application of Construction \ref{normal-family-lift}. 

Let $<$ denote the natural lexicographic ordering of $\mathbb{Z} \times \mathbb{Z}$ where the second factor is cofinal, and $P$ its positive cone.  There is an order-isomorphism $\Psi_1:  (G_1, c_1) \rightarrow Q(\ZZ \times \ZZ, <, (0, n))$ and thus there is an order-isomorphism $\Phi : (\wt{G}_1, <_{c_1}, z_{c_1}) \rightarrow (\ZZ \times \ZZ, <, (0, n))$, it is the composition
\[ (\wt G_1, <_{c_1}, z_{c_1}) \stackrel{L\Psi_1}{\rightarrow}  LQ(\ZZ \times \ZZ, <, (0, n)) \stackrel{\nu_{\ZZ \times \ZZ}}{\rightarrow} (\ZZ \times \ZZ, <, (0, n)).
\]
Here, $\nu_{\ZZ \times \ZZ}$ is the isomorphism arising from the categorical equivalence defined in Proposition \ref{categorical-equivalence} and $Q, L$ refer to the lift and quotient functors of  Lemma \ref{functors}.  

Let $\delta : \ZZ \times \ZZ \rightarrow \ZZ \times \ZZ$ denote the automorphism with matrix $\begin{bsmallmatrix} 1 & 0 \\ n  & 1 \end{bsmallmatrix}$.  One can verify that for each $k \in \ZZ$, the left ordering $<^k$ corresponding to the positive cone $(\delta^k)^*P$ also satisfies $Q(\ZZ \times \ZZ, <^k, (0,n)) \cong (G_1, c_1)$ where $\cong$ is an order-isomorphism.  Set $S = \{ (\delta^k \circ \Phi)^*P \mid k \in \ZZ \}$ and note that $S \subset \wt{\mathcal{R}}_1$, for if $\prec \in S$ then $\eta_GQ(\wt{G}_1, \prec, z_{c_1}) = (G_1, c_1)$ (see Remark \ref{construction-remark}).  Moreover the orderings of $S$ are distinct upon restriction to any rank two abelian subgroup of $\wt{G}_1$, and so $\wt{\phi}_1^*S$ is infinite.  In particular, so is $\wt{\phi}_1^* \wt{\mathcal{R}}_1$. 

On the other hand, if we let $K = \langle x, y \mid xyx^{-1} = y^{-1} \rangle$ then $K$ admits exactly four left orderings, and all arise lexicographically from the short exact sequence
\[ 1 \rightarrow \langle y \rangle \rightarrow K \rightarrow \ZZ \rightarrow 1, 
\]
where the quotient is generated by the image of $x$.  Fixing a left ordering $<$ of $K$ with $y>1$, one can verify that there is an order-isomorphism $\Psi_2 : (G_2, c_2) \rightarrow Q(K, <, x^2) $ so that $K \cong \wt{G}_2$ by reasoning similar to the case of $G_1$.  In particular, because $K$ only has four left orderings, $|\wt{\phi}_2^* \wt{\mathcal{R}}| \leq 4$ and so $\wt{\mathcal{R}}_1$ and $\wt{\mathcal{R}}_2$ cannot be compatible with $\{ \wt{\phi}_1, \wt{\phi}_2\}$.  

This incompatibility arises from the fact that $\phi_2^*:H^1(G_2;\ZZ) \to H^1(H;\ZZ)$ cannot be surjective, since $G_2/G_2'$ is torsion and thus $H^1(G_2;\ZZ)$ is trivial while $H^1(H ;\ZZ)=H^1(\ZZ ;\ZZ)$ is infinite.  Note, however, that the required normal families of left orderings do exist and are constructed in the course of the proof of Proposition \ref{linear-amalgamation}.
\end{example}

Thus we ask the following question:

\begin{question}
Do there exist sufficient conditions on the groups $(G_i, c_i)$ which guarantee the existence of a circular ordering $c$ as in Theorem \ref{main-coro}(1), which make no reference to left orderings of the lifts $\widetilde G_i$?  In particular, is it possible to drop the surjectivity assumption on the first cohomology in Proposition \ref{bg-circ-back}?
\end{question}

\section{Special cases of amalgamation}
\label{special-cases}
Perhaps the most natural corollaries of the theorem of Bludov-Glass are that amalgamation of left-ordered groups along convex subgroups or along rank one abelian subgroups preserves left-orderability.  There are analogous results in the case of circularly-ordered groups, which we prove below.

Recall that a subgroup $H$ of a left-ordered group $(G, <)$ is \textit{convex} if whenever $h_1, h_2 \in H$ and $h_1 < g < h_2$ for some $g \in G$, then $g \in H$.  We recall the generalization to circularly-ordered groups.  Suppose that $H$ is a proper subgroup of a circularly-ordered group $(G, c)$.  Then $H$ is said to be \emph{convex} with respect to the circular ordering $c$ of $G$ if for every $g \in G \setminus H$, $f \in G$ and $h_1, h_2 \in H$, whenever $c(h_1, g, h_2) = 1$ and $c(h_2, f, h_1) = 1$ then $f \in H$ (this is in analogy with an established definition in the case of two-sided invariant circular orderings, see e.g. \cite{JP88}).  We first establish a few elementary results concerning convex subgroups in circularly-ordered groups, some of which appear in \cite{CMR17}.

\begin{lemma} \label{convex}
Suppose that $(G,c)$ is a circularly-ordered group and let $H \subset G$ be a subgroup with $|G:H| \geq 3$.  Then $H$ is convex if and only if the left cosets $G/H$ inherit a circular ordering $\ol{c} : (G/H)^3 \rightarrow \{ 0, \pm1 \}$ defined by $\ol{c}(g_1H, g_2H, g_3H) = c(g_1, g_2, g_3)$ whenever $g_1H, g_2H$, and $g_3H$ are distinct cosets.
\end{lemma}
\begin{proof}
The forward direction is proved in \cite{CMR17}.  Conversely, suppose $H$ is not convex.  That is, there exists $h_1,h_2,g,f \in G$ such that $c(h_1,g,h_2) = 1$ and $c(h_2,f,h_1) = 1$ but $g,f \notin H$.  Then by the cocycle condition, $c(h_1,g,f) = 1$ and $c(h_2,g,f) = -1$.  If $gH$ and $fH$ are distinct cosets, then $\ol{c}$ is not well-defined and we are done.  Suppose not, and choose $t \in G\setminus H$ such that $tH \neq gH$.  If $c(h_1,t,h_2) = 1$, then applying the cocycle condition gives $c(h_1,t,f) = 1$ and $c(h_2,t,f) = -1$ so $\ol c$ is not well-defined.  Similarly, if $c(h_2,t,h_1) = 1$, then $c(h_1,g,t) = 1$ and $c(h_2,g,t) = -1$, completing the proof.
\end{proof}

\begin{lemma}
\label{it-is-a-cone}
Suppose that $H$ is a proper convex subgroup of the circularly-ordered group $(G, c)$.  Then the set 
\[ P = \{ h \in H \mid c(id, h, g) = 1 \mbox{ for some $g \in G \setminus H$} \}
\]
is the positive cone of a left ordering of $H$.  Moreover if $<$ is the left ordering corresponding to $P$, then for every $h_1, h_2, h_3 \in H$ we have $h_1 <h_2 <h_3$ if and only if $c(h_1, h_2, h_3) = 1$ (up to cyclic permutation of the arguments).
\end{lemma}
\begin{proof}
We first check that $P \sqcup P^{-1} \sqcup \{ id \} = H$.  To see this, let $h \in H$ with $h \neq 1$ be given and suppose $g \in G \setminus H$.  Then either $c(id, h, g) = 1$ yielding $h \in P$ or $c(id, g, h) = 1$, yielding $c(id, h^{-1}, h^{-1}g) = 1$ where $h^{-1}g \in G \setminus H$ and thus $h^{-1} \in P$. Therefore $H \subset P \cup P^{-1}$.   Second, suppose that $h \in P \cap P^{-1}$, so that there exist $f, g \in G \setminus H$ such that $c(id, h, g) = 1$ and $c(id, h^{-1}, f) = 1$.  But then, from the first equality, $c(h^{-1}, id, h^{-1}g)=1$.  But since $c(id, h^{-1}, f) = 1$ with $f \in G \setminus H$, by convexity this implies $h^{-1}g \in H$, a contradiction.

Next, to show that $P \cdot P \subset P$, suppose that $h, k \in H$ satisfy $c(id, h, g) =1$ and $c(id, k, f) = 1$ for some $g, f \in G \setminus H$.  If $c(id, hk, g)=1$ we are done, so suppose that $c(id, g, hk) =1$ and note that $c(h, hk, hf) = 1$ as well.  Combining $c(id, g, hk) =1$ and $c(id, h, g) =1$ we have $c(h, g, hk) =1$.  But now $c(h, g, hk) =1$ and $c(hk, hf, h) = 1$ imply that one of $hk$ or $g$ lies in $H$ by convexity, a contradiction.

That $h_1 <h_2 <h_3$ if and only if $c(h_1, h_2, h_3) = 1$ (up to cyclically permuting the arguments of $c$) is a straightforward check using the definition of $P$.
\end{proof}

Thus, when $H$ is a proper convex subgroup of $(G,c)$ we will say that $H$ is \textit{left-ordered by restriction}.  We call the left ordering of $H$ corresponding to the positive cone $P$ of Proposition \ref{it-is-a-cone} ``the left ordering of $H$ arising from the restriction of $c$". Note that for proper convex subgroups $H$ of $(G, c)$ this agrees with \cite[Definition 2.2]{CMR17}, where one says that $H$ is left-ordered by restriction if the set 
\[ Q = \{ h \in H \mid c(h^{-1}, id, h) =1 \}
\]
forms a positive cone.  One can verify that under the assumptions of Proposition \ref{it-is-a-cone}, we have $Q=P$.

We return to left-orderability of free products with amalgamation, and begin with an observation that tells us how convex subgroups behave with respect to the lifting construction.  Suppose that $H$ is a convex subgroup of $(G, c)$ with positive cone $P$ as in Proposition \ref{it-is-a-cone}.  Define a function $d: H \rightarrow \ZZ$ by $d(id) = 0$ and 
\[
   d(h)= 
\begin{cases}
    1& \text{if } h \notin P\\
    0 & \text{if } h \in P.
\end{cases}
\]

This function $d$ satisfies $f_c(g,h) = d(g) -d(gh)+d(h)$ for all $g,h \in H$, meaning that when $H$ is convex, the restriction of $f_c$ to $H$ is a coboundary.  Consequently $\widetilde{H}$ is a split central extension, with an explicit isomorphism $\phi : \ZZ \times H \rightarrow \widetilde{H}$ given by $\phi(n,h) = (n-d(h), h)$.  Via this isomorphism we can identify $\widetilde{H} \subset \widetilde{G}$ with $\ZZ \times H$, embedding $H$ as a subgroup of $\wt G$ and identifying the $\ZZ$ factor with $\langle z_c \rangle$.

One checks that the left ordering of $\ZZ \times H$ arising from the restriction of $<_c$ is lexicographic where $\mathbb{Z}$ is cofinal and $H \subset \wt G$ is equipped with the positive cone $P$ of Proposition \ref{it-is-a-cone}.\footnote{We will always distinguish the subgroup $H \subset G$ from the subgroup $H \subset \widetilde{G}$ by indicating the supergroup whenever confusion may arise.}  

\begin{lemma}
\label{lem-convex-lift}
Suppose that $(G,c)$ is a circularly-ordered group and that $H \subset G$ is a proper convex subgroup.  With notation as above, the image of the inclusion $\iota: H \rightarrow \widetilde{G}$ given by $\iota(h) = (-d(h), h)$ (which we will simply write as $H \subset \wt G$) is a convex subgroup relative to the left ordering $<_c$ of $\widetilde{G}$.
\end{lemma}
\begin{proof}
Denote by $\tilde{h}, \tilde{g}$ arbitrary nonidentity elements of $H \subset \widetilde{G}$ and $\widetilde{G}$ respectively that project to the elements $h, g$ under the projection map $\widetilde{G} \rightarrow G$.  Note that neither of $h$, $g$ is the identity.  It suffices to check that under the assumption $id <_c \tilde{g} <_c \tilde{h}$, we have $\tilde{g} \in H$.

First note that $\tilde{h} \in H \subset \widetilde{G}$ implies $\tilde{h} =(0, h)$.  Then as $id <_c \tilde{g} <_c \tilde{h}$, we know $\tilde{g} = (0,g)$.  Thus $id <_c \tilde{g} <_c \tilde{h}$ implies $id <_c (0,g)^{-1}(0,h) = (f_c(g^{-1}, h)-1, g^{-1}h)$ which happens if and only if $c(id, g, h) =1$.

Now since $id <_c (0,h)$, we know that $h \in P$, where $P$ is the positive cone of Proposition \ref{it-is-a-cone}.  So there exists $x \in G \setminus H$ with $c(id, h, x) =1$.  Combining this with $c(id, g, h) =1$ forces $g \in H$ by convexity, so $(0, g) = \tilde g \in H \subset \wt G$.
\end{proof}

We are now ready to produce the required normal families needed to circularly order amalgamations along convex subgroups.

\begin{lemma}
\label{lem-lex-lifts-normal}
Suppose that $(G,c)$ is a circularly-ordered group and that $H \subset G$ is a proper convex subgroup.  Let $\mathcal{N} \subset \mathrm{LO}(\widetilde{G})$ denote the subset of left orderings of $\widetilde{G}$ that restrict to lexicographic orderings of  $\widetilde{H} \cong \ZZ \times H$ relative to which the $\mathbb{Z}$ factor is cofinal in $\widetilde{G}$.  Then $\mathcal{N}$ is normal.
\end{lemma}
\begin{proof}
Note that an ordering $<$ of $\widetilde{H} \cong \ZZ \times H$ is lexicographic with $\mathbb{Z}$ cofinal if and only if the only non-cofinal elements are exactly the elements of $H$ (i.e. elements in $\ZZ \times H$ of the form $(0,h)$). Thus to prove the lemma it suffices to check that for every $h \in \widetilde{G}$ if $h$ is $<$-cofinal then $h$ is $<^g$-cofinal for all $g \in G$.  We prove this claim next.

In an ordered group $(G, <)$ with positive cofinal central element $z$, the cofinal elements are
\[ \{ h \in G \mid \exists k \in \mathbb{Z} \mbox{ such that } h^k > z \}.
\]
So to prove our claim it suffices to show that for a given $h \in \wt G$, if $z< h^k$ for some $k  \in \mathbb{Z}$ then $z <^g h^{\ell}$ for some $\ell \in \mathbb{Z}$ and for every $g \in \widetilde{G}$.  

To this end, given $g, h \in \widetilde{G}$ as above, note that if $h^k >z$ then $h^{2k}>z^2$.  Also there exists $j \in \mathbb{Z}$ such that $z^j < g \leq z^{j+1}$, from which we calculate $z^{-(j+1)} \leq g^{-1} < z^j$.  Combining these two inequalities with $h^{2k}>z^2$ yields $z < gh^{2k}g^{-1}$, as needed.
\end{proof}

\linamalg*

\begin{proof}
Equip each group $\widetilde{G}_i$ with the normal family $\mathcal{N}_i \subset \mathrm{LO}(\widetilde{G}_i)$ of all orderings of $\widetilde{G}_i$ which restrict to lexicographic orderings of $\widetilde{H}_i \cong \ZZ \times H_i$, as in Lemma \ref{lem-lex-lifts-normal}.   By the remarks preceding Lemma \ref{lem-convex-lift}, $<_{c_i} \in \mathcal{N}_i$ for all $i$.  Moreover the lifts $\{ <_{c_i} \}_{i \in I}$ are compatible with $\{ \widetilde{\phi}_i \}_{i \in I}$, since the order isomorphisms $\phi_i^{-1} \phi_j : (H_i, c_i) \rightarrow (H_j, c_j)$ lift to order isomorphisms
\[
\tilde{\phi}_i^{-1} \tilde{\phi}_j : (\wt H_i, <_{c_i}, z_i) \rightarrow (\wt H_j, <_{c_j}, z_j).
\]

  We will show that $\{ \mathcal{N}_i \}_{i \in I}$ are compatible with $\{ \widetilde{\phi}_i \}_{i \in I}$.  To see this, it suffices to observe that every lexicographic ordering of $\widetilde{H}_i \cong \ZZ \times H_i$ with $\mathbb{Z}$ cofinal arises as the restriction of some ordering of $\widetilde{G}_i$:  This follows from Lemma \ref{lem-convex-lift}, which allows us to extend any left ordering of $H_i \subset \widetilde{G}_i$ to a left ordering of $\widetilde{G}_i$ since $H_i$ is $<_{c_i}$-convex.  Moreover since the generator of $\mathbb{Z}$ appearing in the direct product decomposition of $\widetilde{H}_i \cong \ZZ \times H_i$ is the cofinal central element of $\wt G$, restricting this extension ordering to $\wt H_i$ yields a lexicographic ordering, and so the order is in $\mathcal{N}_i$.   By Theorem \ref{bludov-glass}, Theorem \ref{main-coro}(2) holds.  Thus Theorem \ref{main-coro}(1) holds, completing the proof.
\end{proof}

This previous proposition is readily applicable in a special case relating to Question \ref{3manifold-problem}.

\begin{example}
\label{convex-groups-and-3-manifolds}  Suppose that $M$ is a compact, connected, orientable hyperbolic $3$-manifold with a single torus boundary component $T_M \subset \partial M$.  Let $\Delta_M$ denote the set of cusps in the universal cover of $M$, and note that there is an action of $\pi_1(M)$ on  $\Delta_M$ for which $\pi_1(T_M)$ is the stabilizer of a cusp.  By \cite[Lemma 2.11]{SSpreprint}, provided $M$ admits a certain nice triangulation, one can guarantee the existence of a unique circular ordering $d_M : \Delta_M^3 \rightarrow \{ 0, \pm 1 \}$ that is invariant under the $\pi_1(M)$-action.  By choosing any left ordering we please for the subgroup $\pi_1(T_M)$ and ordering $\pi_1(M)$ lexicographically, we arrive at a circular ordering $c_M$ of $\pi_1(M)$ such that $\pi_1(T_M)$ is convex.  

Let $M$ and $N$ be two $3$-manifolds as above and let $\psi :T_M \rightarrow T_N$ be any homeomorphism identifying their respective boundary tori.  Equip $\pi_1(M)$ and $\pi_1(N)$ with orderings $c_M$ and $c_N$ respectively where the orderings of $\pi_1(T_M)$ and $\pi_1(T_N)$ are chosen so that $\psi$ induces an order-isomorphism between the peripheral subgroups $\pi_1(T_M)$ and $\pi_1(T_N)$. Then as
\[
\pi_1(M \cup_{\psi} N) = \pi_1(M) * \pi_1(N) (\pi_1(T_M) \stackrel{\psi}{\cong} \pi_1(T_N))
\]
and $\psi$ is compatible with the orderings $c_M$ and $c_N$, we conclude that $\pi_1(M \cup_{\psi} N)$ is circularly-orderable with an ordering extending that of each of the factors, by Proposition \ref{linear-amalgamation}.
\end{example}

For the next proposition, we say that a subgroup $H$ of a left-ordered group $G$ with ordering $<$ is $<$-cofinal if there exists $h \in H$ that is $<$-cofinal.

\begin{proposition}
\label{cofinal-rank-two}
Suppose that $(G_i, <_i)$ are left-ordered groups with subgroups $H_i \subset G_i$, each equipped with an isomorphism $\phi_i : H \rightarrow H_i$ for all $i \in I$.  Suppose that $H \subset \mathbb{Q}^2$ is a rank two abelian subgroup and that $H_i$ is $<_i$-cofinal for all $i$ and that $\{<_i \}_{i \in  I}$ are compatible with $\{ \phi_i \}_{i \in I}$.  Then $*_{i \in I} G_i  (H_i \stackrel{\phi_i}{\cong} H)$ admits a left ordering that extends each of the $<_i$.
\end{proposition}
\begin{proof}
Since $\{<_i \}_{i \in  I}$ are compatible with $\{ \phi_i \}_{i \in I}$, there is an ordering $<$ (the pullback of $<_j$ along $\phi_j$ for any $j$) such that $\phi_i : (H,<) \rightarrow (H_i, <_i)$ is an order isomorphism for all $i$.
The ordering $<$ determines a line in $\mathbb{Q}^2$.  All elements to one side of the line positive, the elements to the other side negative.  Depending on whether or not this line has irrational slope, there are two cases:

First, if the line has irrational slope then every nonidentity element of $H_i$ is $<_i$-cofinal for every $i \in I$.  The result then follows from \cite[Corollary 5.8]{BV09}, since the sign of a cofinal element is preserved under conjugation---as in the proof of Lemma \ref{lem-lex-lifts-normal}.  

On the other hand, a line of rational slope can be dealt with as in the proof of \cite[Proposition 11.5]{BC17}. Suppose that in the restriction of $<_i$ to $H_i$, the rank one subgroup $K_i$ is convex.  For each $i$ define $\mathcal{N}_i$ to be the collection of all left orderings of $G_i$ which restrict to $H_i$ in such a way that $K_i$ is convex.  Note that $<_i \in \mathcal{N}_i$ by definition.  Because all elements of $H_i$ which are not in $K_i$ are $<_i$-cofinal, the family $\mathcal{N}_i$ of orderings is normal (c.f. the proof of Lemma \ref{lem-lex-lifts-normal}).   We show compatibility with $\{ \phi_i \}_{i \in I}$ as follows.

Having fixed $K_i \subset H_i$ as above, there are exactly four left orderings of $H_i$ which realize $K_i$ as a convex subgroup.  They are those that arise lexicographically from the short exact sequence
\[ 1 \rightarrow K_i \rightarrow H_i \rightarrow H_i / K_i \rightarrow 1
\]
where the kernel and image are both rank one abelian.  Call this collection $O_i \subset \mathrm{LO}(H_i)$.   There is a restriction map $r_i: \mathrm{LO}(G_i) \rightarrow \mathrm{LO}(H_i)$, and if the condition $r_i(\mathcal{N}_i) = O_i$ holds then the families $\mathcal{N}_i$ will be compatible $\{ \phi_i \}_{i \in I}$, as in  \cite[Proposition 11.5]{BC17}.  

To see that this condition holds, for each $i$, the family of sets
\[
\mathcal{X}_i = \{ S \subset G_i \mid x \in S \mbox{ and } y<x \Rightarrow y\in S \}
\]
is ordered by inclusion, and the natural left-action of $G_i$ on $\mathcal{X}_i$ preserves this order.  Let $S_i$ denote the stabilizer of the set $X_i = \{ x \in G_i \mid x < g \mbox{ for some } g \in K_i \}$, and note that $S_i \cap H_i = K_i$.

Via the usual method of constructing a left ordering on $G_i$ from an order-preserving action on a linearly ordered set $\mathcal{X}_i$, we can construct a left ordering of $G_i$ relative to which $S_i$ is convex.  Using the convex subgroup $S_i$, we can order $G_i$ in four distinct ways, such that each of the four orderings lies in $\mathcal{N}_i$ and the restriction of each to $H_i$ is distinct.  We conclude $r_i(\mathcal{N}_i) = O_i$, and compatibility holds.
\end{proof}

\cycamalg*

%
\begin{proof}
In either case, $\{ <_{c_i} \}_{i \in I}$ are compatible with $\{ \widetilde{\phi}_i \}_{i \in I}$.  With this in hand we consider the cases separately.

If $(H,d)$ is a subgroup of the rational points of $S^1$, then the lift $\widetilde{H}$ is a subgroup of $\mathbb{Q}$.  Consequently $*_{i\in I} \widetilde{G}_i (\widetilde{H}_i \stackrel{\widetilde \phi_i}{\cong}\widetilde{H})$ is a free product of left-ordered groups with amalgamation along rank one abelian subgroups, and thus admits a left ordering extending the ordering of each of the factors \cite[Corollary 5.3]{BV09}.  By Theorem \ref{main-coro}, $*_{i \in I} G_i  (H_i \stackrel{\phi_i}{\cong} H)$ admits a circular ordering that extends each of the $c_i$.

Now suppose $(H,d)$ is $\QQ$ or $\ZZ$ with the ordering above, then the lifts $\wt H_i$ are isomorphic to $H_i \times \ZZ$, each left-ordered so that the $\mathbb{Z}$ factor is $<_{c_i}$-cofinal.  By Proposition \ref{cofinal-rank-two},
\[
*_{i\in I} \widetilde{G}_i (\widetilde{H}_i \stackrel{\widetilde \phi_i}{\cong}\widetilde{H})
\]
is left-orderable by an ordering extending each of the lifted orderings $<_{c_i}$. By Theorem \ref{main-coro}, $*_{i \in I} G_i  (H_i \stackrel{\phi_i}{\cong} H)$ admits a circular ordering that extends each of the $c_i$.
\end{proof}

\begin{example}
 \label{CO-cyclic-amalgamation-failure}
Note that Proposition \ref{cyclic-amalgamation} does not imply that amalgamating circularly-orderable groups along cyclic subgroups yields a circularly-orderable result, unlike the case of left-orderable groups.  The primary difference is that if $H_i$ are subgroups of left-orderable groups $G_i$ for $i \in I$, and $\phi_i : H \rightarrow H_i$ are isomorphisms, then there is always a choice of left orderings on $H$ and the groups $G_i$ such that every map $\phi_i : H \rightarrow H_i$ is order-preserving. 

In contrast, suppose that $G_1, G_2$ are both isomorphic to $\mathbb{Q}/\mathbb{Z}$, and thus each admits precisely two circular orderings---the restriction of the standard circular ordering of $S^1$, and its reverse (the circular ordering obtained by multiplying the standard ordering by $-1$).  Fix a prime $p \geq 5$ and let $\phi_i : \ZZ/p\ZZ \rightarrow \QQ/\ZZ$ denote the map determined by the assignment $\phi_i(1) = \frac{i}{p}+ \ZZ$ for $i=1,2$.  Let $H_1, H_2$ denote the copies of $\ZZ/p\ZZ$ contained in $G_1$ and $G_2$ that are generated by the image of $\frac{1}{p}$ under the quotient $\QQ \rightarrow \QQ/\ZZ$, and consider the free product with amalgamation $G_1 * G_2(H \stackrel{\phi_i}{\cong} H_i)$.  This group is not circularly-orderable, as there are no circular orderings on $H$, $G_1$ and $G_2$ that make $G_1 \stackrel{\phi_1}{\leftarrow} H \stackrel{\phi_2}{\rightarrow} G_2$ into an amalgamation diagram in $\Circ$.  

At first blush this may appear at odds with our main theorem, as $\widetilde G_i \cong \QQ$, $\wt H_i \cong \ZZ$ and so for any choice of isomorphisms $\wt \phi_i: \ZZ \rightarrow H_i$ the group $\wt G_1*\wt G_2(\ZZ \stackrel{\wt \phi_i}{\cong} \wt H_i)$ will be left-orderable by \cite[Corollary 5.2]{BV09}.  The key observation, however, is that any identification of $\QQ$ with $\wt G_i$ implicitly involves making a choice $q_i \in \QQ$ of cofinal central element, and the diagram $\wt G_1 \stackrel{\wt \phi_1}{\leftarrow} \ZZ \stackrel{\wt \phi_2}{\rightarrow} \wt G_2$ will only pass to an amalgamation diagram in $\Circ$ if $\wt \phi_2 \wt\phi_1^{-1}(q_1) = q_2$.  Imposing this condition on the maps $\wt\phi_i$ means precisely that upon passing to quotients, the standard circular orderings of $H_i \subset G_i$ will be compatible with $\phi_2 \phi_1^{-1}$---in particular, the setup of the previous paragraph can never arise as such a quotient.
\end{example}

\section{$\mathbf{Circ}$ as a tensor category}
\label{tensor-cat}

This section explores ideas first put forward by Rolfsen \cite{Rolfsen17}. We show how a certain explicit construction of a circular ordering on the free product produces a tensor structure on $\Circ$. 

We begin by reviewing a construction from \cite{BS15} which provides an explicit circular ordering of the free product $G*H$.  Let $(G, c_G)$ and $(H_, c_H)$ be circularly-ordered groups.  First, we define what it means for a triple $(w_1, w_2, w_3) \in (G*H)^3$ to be \textit{reduced}.  Consider the following three reduction operations that one can perform on such a triple:

\begin{enumerate}
\item Suppose $x \in G \cup H$ is the leftmost letter of all three words $w_1, w_2, w_3$.   That is, $w_1 = xw_1'$ and $w_2 = xw_2'$ and $w_3 = xw_3'$, and all of the $xw_i'$'s are reduced words.  In this case, replace $(w_1, w_2, w_3)$ with $(w_1', w_2', w_3')$. 
\item Suppose $x \in G \cup H$ appears as the leftmost letter in exactly two of the words $\{ w_1, w_2, w_3 \}$.  Then left-multiply the triple by $x^{-1}$.  The word $w_i$ which does not have $x$ as its leftmost letter is thus replaced with $x^{-1}w_i$ in the triple $(w_1, w_2, w_3)$, which may not be a reduced word.  Thus to complete the operation, we reduce $x^{-1}w_i$.
\item Suppose $x \in G \cup H$ is the leftmost letter of exactly one of $\{ w_1, w_2, w_3 \}$, say $w_i$.  Then replace $w_i$ with $x$.
\end{enumerate}

Call a triple $(x,y,z) \in   (G*H)^3$ a \textit{reduction} of $(w_1, w_2, w_3) \in (G*H)^3$ if one can arrive at $(x,y,z)$ starting from $(w_1, w_2, w_3)$ by performing a series of the moves (1)--(3) above, and if no further moves can be performed on the triple $(x,y,z)$ we call it a \textit{minimal reduction}.  By \cite[Proof of Theorem 4.3]{BS15} every triple $(w_1, w_2, w_3) \in (G*H)^3$ admits a unique minimal reduction.  Moreover, if $(x,y,z)$ is a minimal reduction of $(w_1, w_2, w_3) \in (G*H)^3$, then either exactly two of $\{x, y, z \}$ lie in $G$ while the other is in $H$, or exactly two of $\{x, y, z \}$ lie in $H$ while the other is in $G$, or $\{x, y, z \} \subset H$, or $\{x, y, z \} \subset G$.   That is, the minimal reduction always lies in $(G \cup H)^3$.

We are now ready to state the result of \cite[Theorem 4.3]{BS15}, which defines a circular ordering of the free product of two circularly-ordered groups $(G, c_G)$ and $(H, c_H)$.  Define $c: (G*H)^3 \rightarrow \{ \pm 1, 0\}$ according to the rules:
\begin{enumerate}
\item We insist that $c$ is invariant under cyclic permutation of its arguments, and that $c(g, h, id ) = +1$ and $c(h,g,id) =-1$ for all $g \in G \setminus \{ id \}$ and $h \in H \setminus \{ id \}$.
\item On $G^3$ and $H^3$, define $c$ by $c|_{G^3} = c_G$ and $c|_{H^3} = c_H$.
\item Define $c(g_1, g_2, h) = c_G(g_1, g_2, id)$ and $c(g, h_1, h_2) = c_H(id, h_1, h_2)$ for all $g, g_1, g_2 \in G\setminus \{ id \}$ and $h, h_1, h_2 \in H\setminus \{ id \}$.  Use (1) to extend this to all of $(G \cup H)^3$.
\item If $(x,y,z)$ is the minimal reduction of $(w_1, w_2, w_3) \in (G*H)^3$ then 
\[ c(w_1, w_2, w_3) = c(x,y,z).
\]
\end{enumerate}
Such a $c$ exists and is uniquely determined by these conditions.

We are now ready to state our result, which mirrors a result of Rolfsen \cite[Theorem 1]{Rolfsen17} in the case of left- and bi-orderability of free products.  It also further illustrates the necessity of restricting our attention to injective homomorphisms in our definition of $\Circ$.

For the next proposition, a {\it faux order-preserving homomorphism} $\phi:(G,c) \to (H,d)$ is a homomorphism such that $|c(g_1,g_2,g_3) - d(\phi(g_1),\phi(g_2),\phi(g_3))| \leq 1$ for all $g_1,g_2,g_3 \in G$.  Such a homomorphism has the property that $c(g_1,g_2,g_3) = d(\phi(g_1),\phi(g_2),\phi(g_3))$ if $(\phi(g_1),\phi(g_2),\phi(g_3)) \notin\Delta(H)$, so it is the appropriate definition of order-preserving while allowing for non-injective homomorphisms.

\begin{proposition}
\label{functor-theorem}
Suppose that $(G_i, c_i)$ and $(H_i, d_i)$ are circularly-ordered groups for $i=1, 2$, and let $(G_1*G_2, c)$ and $(H_1*H_2, d)$ denote the free products with circular orderings constructed as above.  
\begin{enumerate}
\item  If $\phi_i : (G_i, c_i) \rightarrow (H_i, d_i)$ are order-preserving homomorphisms, then the homomorphism $\phi_1 * \phi_2 : (G_1 *G_2, c) \rightarrow (H_1 * H_2, d)$ is order-preserving.
\item If one of $\phi_i : (G_i, c_i) \rightarrow (H_i, d_i)$ is a non-injective faux order-preserving homomorphism, then the homomorphism $\phi_1 * \phi_2 : (G_1 *G_2, c) \rightarrow (H_1 * H_2, d)$ is not a faux order-preserving homomorphism.
 \end{enumerate}
\end{proposition}
\begin{proof}
Set $\phi = \phi_1*\phi_2$ and let $(w_1, w_2, w_3) \in (G_1*G_2)^3$ be given, suppose both $\phi_1$ and $\phi_2$ are order-preserving.  First note that if $(x, y, z) \in (G_1 \cup G_2)^3$ is the minimal reduction of $(w_1, w_2, w_3) \in (G_1*G_2)^3$, then $(\phi(x), \phi(y), \phi(z))$ is the minimal reduction of $(\phi(w_1), \phi(w_2), \phi(w_3))$.  Indeed, this follows from the observation that since $\phi$ is injective, the triple $(w_1, w_2, w_3)$ admits an operation of type (1), (2) or (3) if and only if the triple $(\phi(w_1), \phi(w_2), \phi(w_3))$ admits an operation of the same type. Thus 
\[
d(\phi(w_1), \phi(w_2), \phi(w_3)) = d(\phi(x), \phi(y), \phi(z)),
\]
and we consider cases. 

\noindent \textbf{Case 1.} There exists $i$ such that $\{ \phi(x), \phi(y), \phi(z) \} \subset H_i$.  Then 
\[d(\phi(x), \phi(y), \phi(z)) = d_i(\phi_i(x), \phi_i(y), \phi_i(z)) = c_i(x, y, z) = c(w_1, w_2, w_3)
\]
and so $d(\phi(w_1), \phi(w_2), \phi(w_3)) = c(w_1, w_2, w_3)$.

\noindent \textbf{Case 2.1}  Two of $\{ \phi(x), \phi(y), \phi(z) \}$ are contained in $H_1\setminus \{ id \}$, the other is contained in $H_2 \setminus \{ id \}$.  Without loss of generality suppose that $\phi(x), \phi(y) \in H_1$ and $\phi(z) \in H_2$.    Then 
\[ d(\phi(x), \phi(y), \phi(z)) = d_1(\phi_1(x), \phi_1(y), id) = c_1(x, y, id) = c(x, y, z).
\]
Thus $d(\phi(w_1), \phi(w_2), \phi(w_3)) = c(w_1, w_2, w_3)$.

\noindent \textbf{Case 2.2.} Two of $\{ \phi(x), \phi(y), \phi(z) \}$ are contained in $H_2\setminus \{ id \}$, the other is contained in $H_1 \setminus \{ id \}$.  Proceed as in the previous case.

\noindent \textbf{Case 3.}  One of $\{ \phi(x), \phi(y), \phi(z) \}$ is equal to the identity, one lies in $H_1$, and the other in $H_2$.    Without loss of generality suppose that $\phi(x) \in H_1, \phi(y) \in H_2$, and $\phi(z) =id$, any other combination can be dealt with via cyclic permutation or appropriate change of sign.   As each of $\phi_i$ is injective, we know $z=id$.  Thus we calculate
\[ d(\phi(x), \phi(y), \phi(z) ) = d(\phi(x), \phi(y), id) = +1,
\]
and
\[ c(x, y, z) = c(x, y, id) = +1.
\]
Thus $d(\phi(w_1), \phi(w_2), \phi(w_3)) = c(w_1, w_2, w_3)$, proving (1).

To prove (2), suppose that $\phi_2$ is a non-injective faux order-preserving homomorphism (the case of a non-injective $\phi_1$ being similar).  Choose elements $g_2, g_3 \in G_2\sm\{id\}$ with $\phi_2(g_3) = id$, $\phi_2(g_2) \neq id$ and $c_2(id, g_2, g_3) = -1$.   Such a choice of $g_2$ and $g_3$ is always possible:  suppose that initially, one chooses $g_2$ and $g_3$ with $\phi_2(g_3) = id$ and $\phi_2(g_2) \neq id$ yet they satisfy $c_2(id, g_2, g_3) = +1$.  It is easy to check, using left invariance, that $c_2(id, g_3^{-1}g_2, g_3^{-1}) = -1$, so replacing $g_2$ with $g_3^{-1}g_2$ and $g_3$ with $g_3^{-1}$ yields a choice of $g_2$ and $g_3$ which meets our requirements.
 
Let $g_1 \in G_1$ be any element which is not mapped to the identity by $\phi_1$.  Then 
\[
c(g_1, g_2, g_3) = c_2(id, g_2, g_3) = -1.
\]
On the other hand, applying $\phi = \phi_1 * \phi_2$ to the entries of the triple $(g_1, g_2, g_3)$ we arrive at $(\phi_1(g_1), \phi_2(g_2), id)$ and compute
\[ d(\phi_1(g_1), \phi_2(g_2), id)  = +1.
\]
Thus $\phi$ is not faux order-preserving.
  \end{proof}

In a more sophisticated language, Proposition \ref{functor-theorem} establishes that the map
\[
\otimes : \mathbf{Circ} \times\mathbf{Circ} \rightarrow \mathbf{Circ}
\]
defined by $(G_1, c_1) \otimes (G_2, c_2) = (G_1*G_2, c)$ yields a bifunctor, while the same recipe for constructing circular orderings of free products does \textit{not} yield a bifunctor if we do not insist on injectivity of the maps in $\Circ$. 

\begin{theorem}
Equipped with the bifunctor $\otimes$ the category $\mathbf{Circ}$ becomes a tensor category.
\end{theorem}
\begin{proof}
The trivial group with trivial ordering provides the necessary identity.  Given circularly-ordered groups $(G_1, c_1), (G_2, c_2), (G_3, c_3)$ there is a natural isomorphism of groups $(G_1 *G_2)*G_3 \cong G_1*(G_2*G_3)$.  We need to check that the above construction of a circular ordering of these respective groups is associative.  To see this, we introduce the notation $c_{1,2}$, $c_{2,3}$, $c_{(1, 2),3}$ and $c_{1 ,(2,3)}$ to denote the orderings of $G_1 *G_2$, $G_2*G_3$, $(G_1 *G_2)*G_3$ and $G_1*(G_2*G_3)$ that arise from the construction given above.  Let $(w_1, w_2, w_3)$ be any triple of words in the alphabet $G_1 \cup G_2 \cup G_3$.  If the $w_i$ are not all distinct then it is easy to confirm that $c_{(1, 2),3}(w_1, w_2, w_3) = c_{1 ,(2,3)}(w_1, w_2, w_3)$, so assume they are all distinct.

By left-multiplying appropriately, we may replace $(w_1, w_2, w_3)$ with a new triple $(g_1v_1, g_2v_2, g_3v_3)$ where $v_i$ are words in the alphabet $G_1 \cup G_2 \cup G_3$ and the $g_i$ are distinct.  Moreover, by left-invariance 
\[c_{(1, 2),3}(w_1, w_2, w_3) = c_{(1, 2),3}(g_1v_1, g_2v_2, g_3v_3) \quad \text{and} \quad c_{1 ,(2,3)}(w_1, w_2, w_3) = c_{1 ,(2,3)}(g_1v_1, g_2v_2, g_3v_3).
\]

Now since the $g_i$ are distinct, the only permissible operations which reduce $(g_1v_1, g_2v_2, g_3v_3)$ are of the third type (this is true whether we reduce in $(G_1* G_2) * G_3$ or $G_1 * (G_2 *G_3)$).  In the case that $g_i$ are all distinct (say $g_i \in G_i$ for all $i$), then we compute (via operations of type (3), with reductions taking place in the groups indicated by the subscripts):
\[ c_{(1, 2),3}(g_1v_1, g_2v_2, g_3v_3) = c_{(1, 2),3}(g_1v'_1, g_2v'_2, g_3) = c_{1,2}(g_1v'_1, g_2v'_2, id) = c_{1,2}(g_1, g_2, id) = 1.
\]
where $v_1', v_2' \in G_1 * G_2$.  Similarly
\[
c_{1 ,(2,3)}(g_1v_1, g_2v_2, g_3v_3) = c_{1 ,(2,3)}(g_1, g_2v'_2, g_3v'_3) = c_{2,3}(id, g_2v'_2, g_3v'_3) = c_{2,3}(id, g_2, g_3) =1.
\]
where $v_1', v_2' \in G_2 * G_3$.  Thus in thus case, $c_{(1, 2),3}$ and $c_{1 ,(2,3)}$ agree.  

When the $g_i$ all lie in distinct groups $G_i$ and the ordering of the triple $(g_1, g_2, g_3)$ differs from the previous case by an odd permutation in $S_3$, we find via similar computation that $c_{(1, 2),3}$ and $c_{1 ,(2,3)}$ agree.  

Next, when all $g_i$ lie in the same group, say $G_j$, we find (after reductions in the appropriate free products) that 
\[ c_{(1, 2),3}(g_1v_1, g_2v_2, g_3v_3) = c_j(g_1, g_2, g_3) = c_{1 ,(2,3)}(g_1v_1, g_2v_2, g_3v_3).
\]

The cases when exactly two $g_i$ lie in the same $G_j$ are similar in that they reduce to the value of $c_j$ on a particular triple, and we leave them to the reader.  We conclude that $c_{1, (2,3)} = c_{(1,2), 3}$.

Last we check the coherence conditions.  As in the case of \cite[Theorem 8]{Rolfsen17}, this amounts to observing that for circularly-ordered groups $(G_i, c_i)$ with $i=1, 2, 3, 4$, our constructed circular orderings of the groups 
\[ ((G_1 * G_2)*G_3)*G_4, (G_1*(G_2*G_3))*G_4, (G_1*G_2)*(G_3*G_4), 
\]
\[G_1*(G_2*(G_3*G_4)) \mbox{ and } G_1*((G_2*G_3)*G_4)
\]
are identical, by associativity of the construction.
\end{proof}

We finish with the following corollary, which is an immediate consequence of the equivalence of the categories $\Circ$ and $\LOs$.

\begin{corollary}
The category $\LOs$ is a tensor category.
\end{corollary}

\bibliographystyle{plain}
\bibliography{Bibliography}


\end{document}